\documentclass[11pt,reqno]{amsart}
\usepackage{indentfirst,enumerate,cite,amssymb,amsfonts,amsmath,amsthm,mathrsfs,dsfont}
\usepackage{color}
\usepackage{graphicx}
\usepackage{multicol}
\usepackage[colorlinks=true,linkcolor=blue,citecolor=blue]{hyperref}
\usepackage{enumerate}
\usepackage{geometry}
\numberwithin{equation}{section}

\newtheorem{theorem}{Theorem}[section]
\newtheorem{definition}[theorem]{Definition}

\newtheorem{lemma}[theorem]{Lemma}
\newtheorem{proposition}[theorem]{Proposition}
\newtheorem{corollary}[theorem]{Corollary}
\newtheorem{remark}[theorem]{Remark}

\allowdisplaybreaks

\begin{document}
	
	\title[Global Solvability on Non-Compact Manifolds]{Global Solvability for Involutive Systems on Non-Compact Manifolds}
	
	\author[S. Coriasco]{Sandro Coriasco}
	\address{
		Dipartimento di Matematica ``Giuseppe Peano'',
		Università degli Studi di Torino,
		Via Carlo Alberto 10, 10123 Torino,
		Italia
	}
	\email{sandro.coriasco@unito.it}
	
	\author[A. Kirilov]{Alexandre Kirilov}
	\address{
		Departamento de Matemática,
		Universidade Federal do Paraná,
		Caixa Postal 19096, 81531-980 Curitiba, PR,
		Brasil
	}
	\email{akirilov@ufpr.br}
	
	\author[W. de Moraes]{Wagner A. A. de Moraes}
	\address{
		Departamento de Matemática,
		Universidade Federal do Paraná,
		Caixa Postal 19096, 81531-980 Curitiba, PR,
		Brasil
	}
	\email{wagnermoraes@ufpr.br}
	
	\author[P. Tokoro]{Pedro M. Tokoro}
	\address{
		Programa de Pós-Graduação em Matemática,
		Universidade Federal do Paraná,
		Caixa Postal 19096, 81531-980 Curitiba, PR,
		Brasil
	}
	\email{pedro.tokoro@ufpr.br}
	



	\begin{abstract}
		We establish necessary and sufficient conditions for the closedness of the range of a class of first-order differential operators associated with an involutive structure on $M\times\mathbb{T}^m$, where $M$ is a non-compact manifold satisfying suitable geometric assumptions and $\mathbb{T}^m$ is the $m$-dimensional torus. In addition, we prove that a weaker notion of global hypoellipticity ensures the closedness of the range for differential operators on smooth paracompact manifolds, thereby extending to the non-compact setting a result previously obtained by G.~Araújo, I.~Ferra, and L.~Ragognette [J. Anal. Math. 148, No. 1, 85-118, 2022] for compact manifolds.
	\end{abstract}

	\maketitle
	
	\section{Introduction}\label{sec:intro}
	
	In this work, we investigate the closedness of the range of a class of first-order linear partial differential operators acting on the product manifold $M\times\mathbb{T}^m$, where $M$ is a smooth, non-compact, $n$-dimensional manifold diffeomorphic to the interior of a compact manifold with boundary, and $\mathbb{T}^m\simeq \mathbb{R}^m / 2\pi\mathbb{Z}^m$ is the $m$-dimensional torus.  
	More precisely, we consider the operators
	\begin{equation}\label{Lq_smooth}
		\mathbb{L}^q :
		\mathsf{\Lambda}^q C^\infty(M\times\mathbb{T}^m)
		\longrightarrow
		\mathsf{\Lambda}^{q+1} C^\infty(M\times\mathbb{T}^m),
		\qquad
		\mathbb{L}^q u = \mathrm{d}_t u + \sum_{k=1}^{m}\omega_k \wedge \partial_{x_k}u,
	\end{equation}
	for $q=0,\dots,n-1$, where $\mathrm{d}_t$ denotes the exterior derivative on $M$, $\partial_{x_k}$ are the partial derivatives with respect to the periodic variables, and $\omega_1,\dots,\omega_m$ are smooth closed $1$-forms on $M$.
	
	Such operators are naturally associated with involutive structures on manifolds and can be regarded, in local coordinates, as systems of vector fields. We refer the reader to \cite{BCH_book,Treves} for a detailed discussion.
	
	By a standard argument (see, for instance, \cite{Treves}), the transpose of $\mathbb{L}^q$ is given by
	\begin{equation}\label{Lq_transp}
		(-1)^{q+1}\mathbb{L}^{n-q-1} :
		\mathsf{\Lambda}^{n-q-1}\mathcal{E}'(M\times\mathbb{T}^m)
		\longrightarrow
		\mathsf{\Lambda}^{n-q}\mathcal{E}'(M\times\mathbb{T}^m),
	\end{equation}
	and the topological dual of $\mathsf{\Lambda}^q C^\infty(M\times\mathbb{T}^m)$ can be identified with
	$\mathsf{\Lambda}^{n-q-1}\mathcal{E}'(M\times\mathbb{T}^m)$, the space of compactly supported $(n-q-1)$-currents.
	We show that each space $\mathsf{\Lambda}^q C^\infty(M\times\mathbb{T}^m)$ is a Fréchet–Schwartz (FS) space, so its dual
	$\mathsf{\Lambda}^q \mathcal{E}'(M\times\mathbb{T}^m)$ is a dual FS (DFS) space.
	Within this framework, we consider the following notions of global solvability.
	
	\begin{definition}
		We say that $\mathbb{L}^q$ is:
		\begin{itemize}
			\item \emph{Globally $(\mathcal{E}',\mathcal{E}')$-solvable} if, for every
			$f\in\mathsf{\Lambda}^{0,q+1}\mathcal{E}'(M\times\mathbb{T}^m)$ satisfying
			\begin{equation}\label{solv_EE}
				\langle f,\phi\rangle=0,\quad
				\forall \phi\in
				\mathsf{\Lambda}^{0,n-q-1}C^\infty(M\times\mathbb{T}^m)\cap\ker\mathbb{L}^{n-q-1},
			\end{equation}
			there exists $u\in\mathsf{\Lambda}^{0,q}\mathcal{E}'(M\times\mathbb{T}^m)$ such that $\mathbb{L}^q u=f$;
			
			\item \emph{Globally $(C^\infty,C^\infty)$-solvable} if, for every
			$f\in\mathsf{\Lambda}^{0,q+1}C^\infty(M\times\mathbb{T}^m)$ satisfying
			\begin{equation}\label{solv_CC}
				\langle v,f\rangle=0,\quad
				\forall v\in
				\mathsf{\Lambda}^{0,n-q-1}\mathcal{E}'(M\times\mathbb{T}^m)\cap\ker\mathbb{L}^{n-q-1},
			\end{equation}
			there exists $u\in\mathsf{\Lambda}^{0,q}C^\infty(M\times\mathbb{T}^m)$ such that $\mathbb{L}^q u=f$;
			
			\item \emph{Globally $(\mathcal{E}',C_c^\infty)$-solvable} if, for every
			$f\in\mathsf{\Lambda}^{0,q+1}C_c^\infty(M\times\mathbb{T}^m)$ satisfying \eqref{solv_EE}, there exists
			$u\in\mathsf{\Lambda}^{0,q}\mathcal{E}'(M\times\mathbb{T}^m)$ such that $\mathbb{L}^q u=f$.
		\end{itemize}
	\end{definition}
	
	Since the involved spaces are FS/DFS, the first two notions of solvability are equivalent to the closedness of the range of the corresponding operators.
	By duality, the $(C^\infty,C^\infty)$-global solvability of $\mathbb{L}^0$ is equivalent to the
	$(\mathcal{E}',\mathcal{E}')$-global solvability of $\mathbb{L}^{n-1}$, and this, in turn, implies
	the $(\mathcal{E}',C_c^\infty)$-global solvability of $\mathbb{L}^{n-1}$.
	
	Following \cite{ADL2023gs} and \cite{AFJR2024}, we show that the $(C^\infty,C^\infty)$-global solvability of $\mathbb{L}^0$ 	is governed by a number-theoretic property of the family
	$\boldsymbol{\omega}=(\omega_1,\dots,\omega_m)$,
	which is equivalent to a Diophantine condition on a matrix, which we denote by $A(\boldsymbol{\omega})$, determined by the first de~Rham cohomology group of $M$.
	When $M$ is compact, this relation follows from the Hodge theorem.
	If $M$ is diffeomorphic to the interior of a compact manifold with boundary and endowed with a suitable scattering metric,
	a corresponding version of the Hodge theorem still holds under additional assumptions on the forms $\omega_k$,
	so the same analysis applies.  
	In this case, $\dim H_{\mathrm{dR}}^1(M)<\infty$.
	
	We prove that the $(\mathcal{E}',C_c^\infty)$-global solvability of $\mathbb{L}^{n-1}$
	implies the Diophantine condition through a Hörmander-type a priori estimate.
	By expanding in partial Fourier series, we decompose both the spaces of forms and the operator $\mathbb{L}^0$
	into a direct sum indexed by disjoint frequency subsets and show that each component has closed range.
	One component is topologically conjugate to the partial exterior derivative $\mathrm{d}_t$,
	which has closed range by the de~Rham theorem together with functorial properties of exact sequences of nuclear Fréchet spaces.
	For the remaining component, denoted by $\mathbb{L}^0_{\Gamma_1}$,
	we show that the corresponding Diophantine condition ensures global hypoellipticity,
	as established in~\cite{CKMT}, extending results obtained in \cite{ADL2023gh} and \cite{BCM1993} for closed manifolds.
	
	Recall that a differential operator $P:C^\infty(M,\mathbb{E})\to C^\infty(M,\mathbb{F})$ between smooth sections of vector bundles over a smooth manifold $M$ is globally hypoelliptic if
	\[
	u\in\mathscr{D}'(M,\mathbb{E}),\quad Pu\in C^\infty(M,\mathbb{F}) \ \Rightarrow\
	u\in C^\infty(M,\mathbb{E}),
	\]
	and almost globally hypoelliptic if
	\[
	u\in\mathscr{D}'(M,\mathbb{E}),\quad Pu\in C^\infty(M,\mathbb{F}) \ \Rightarrow\
	\exists v\in C^\infty(M,\mathbb{E})\text{ such that }Pu=Pv.
	\]
	For compact $M$, it was proved in \cite{AFR2022} that almost globally hypoelliptic operators have closed range.
	We extend this result to any smooth paracompact manifold $M$, thereby obtaining the closedness of the range of $\mathbb{L}^0_{\Gamma_1}$. To do so, we show that $C^\infty(M)$ can be realized as the projective limit of local Sobolev spaces, which are nested Fréchet spaces with compact embeddings, and apply Komatsu’s argument~\cite{Komatsu1967} to construct an equivalent nested sequence of Banach spaces with compact inclusions. Then, using the abstract framework developed in~\cite{AFR22proc}, we establish a general criterion linking almost global hypoellipticity and closedness of range for maps between pairs of topological vector spaces.
	
	Our two main results can be summarized as follows.
	
	\begin{theorem}\label{thm:main1}
		Let $M$ be a smooth manifold diffeomorphic to the interior of a compact manifold with boundary,
		and let $\boldsymbol{\omega}=(\omega_1,\dots,\omega_m)$ be a family of smooth real-valued closed $1$-forms on $M$
		satisfying suitable assumptions.
		The following statements are equivalent:
		\begin{enumerate}
			\item $\mathbb{L}^0$ is globally $(C^\infty,C^\infty)$-solvable;
			\item $\mathbb{L}^{n-1}$ is globally $(\mathcal{E}',\mathcal{E}')$-solvable;
			\item $\mathbb{L}^{n-1}$ is globally $(\mathcal{E}',C_c^\infty)$-solvable;
			\item $\boldsymbol{\omega}$ is not $\Gamma_1$-Liouville;
			\item The matrix of periods $A(\boldsymbol{\omega})$ satisfies the Diophantine condition~\eqref{DCg} for $\Gamma=\Gamma_1$.
		\end{enumerate}
	\end{theorem}
	
		In the statement of Theorem \ref{thm:main1}, $\Gamma_1\subset\mathbb{Z}^m$ is a distinguished subset of frequencies (see \eqref{eq:Gamma01} below),
		$\Gamma_1$-Liouville families of $1$-forms correspond to those satisfying a number-theoretic approximation property
		(see Definition~\ref{def_gLiou} below),
		and~\eqref{DCg} is the associated Diophantine condition for the matrix $A(\boldsymbol{\omega})$.
	
	\begin{theorem}
		Let $P:C^\infty(M,\mathbb{E})\to C^\infty(M,\mathbb{F})$
		be an almost globally hypoelliptic operator between smooth sections of vector bundles
		over a smooth paracompact manifold $M$.
		Then $P$ has closed range.
	\end{theorem}
	
	The paper is organized as follows. In Section \ref{sec:fsonmfs} we recall 
	some preliminaries and topological properties of functional spaces on a paracompact manifold,
	while in Section \ref{sec:glregcycles} we summarize the hypoellipticity results we obtained in \cite{CKMT},
	which we employ in the sequel.
	Section \ref{sec:regclrng} is devoted to proving the closedness of the range of operators on FS spaces, under certain assumptions.
	Our main results are then proved in the concluding Section \ref{sec:globsol}.

	\section{Function Spaces on Manifolds}\label{sec:fsonmfs}
	
	In this section, we recall some basic facts about the topology of spaces of smooth functions and smooth differential forms on a paracompact manifold.
	
	Let $U \subset \mathbb{R}^n$ be an open set. The space $C^\infty(U)$ of smooth, complex-valued functions on $U$ is a Fréchet space when endowed with the topology generated by the family of seminorms
	\[
	\|f\|_{m,K} = \sup_{t \in K} \sup_{|\alpha| \leq m} \big|\partial_t^\alpha f(t)\big|,
	\]
	where $K$ ranges over all compact subsets of $U$, $m \in \mathbb{N}_0$, and $\alpha \in \mathbb{N}_0^n$ is a multi-index.
	
	Now, let $M$ be a smooth, $n$-dimensional, paracompact manifold, and let $\{(\Omega_j, \chi_j)\}_{j \in \mathbb{N}}$ be a countable, locally finite atlas of $M$, with each $\Omega_j$ having compact closure. The space $C^\infty(M)$ consists of all functions $f: M \to \mathbb{C}$ such that, for every chart $(\Omega_j, \chi_j)$, the local representation
	\[
	f \circ \chi_j^{-1} : \chi_j(\Omega_j) \to \mathbb{C}
	\]
	is smooth in the usual sense.
	
	The canonical locally convex topology on $C^\infty(M)$ is defined by the family of seminorms $\|\cdot\|_{\ell,K}$, where $\ell \in \mathbb{N}_0$ and $K \subset M$ is compact. For any such compact $K$, there exists a finite subset of indices $J_K \subset \mathbb{N}$ such that $\{\Omega_j\}_{j \in J_K}$ covers $K$. We then set
	\begin{equation}\label{semi_lK}
		\|f\|_{\ell,K} = \sum_{j \in J_K} \| f \circ \chi_j^{-1}\|_{\ell, \chi_j(K \cap \Omega_j)}, \quad f \in C^\infty(M).
	\end{equation}
	
	This topology does not depend on the particular choice of finite covering of $K$ nor on the initial atlas of $M$.

	\subsection{Partial Fourier Expansion on \( M \times \mathbb{T}^m \)}\label{sec:partial_fourier} \
	
	Let \( U \subset \mathbb{R}^n \) be an open set, and consider the $m$-dimensional torus \( \mathbb{T}^m = \mathbb{R}^m / 2\pi\mathbb{Z}^m \). Given a function \( f \in C^\infty(U \times \mathbb{T}^m) \) and \( \xi \in \mathbb{Z}^m \), we define its \emph{partial Fourier coefficient} as the function \( \widehat{f}_\xi \in C^\infty(U) \) given by
	\[
	\widehat{f}_\xi(t) = \frac{1}{(2\pi)^m} \int_{\mathbb{T}^m} e^{-i\xi\cdot x} \, f(t,x) \,\mathrm{d}x, \quad t \in U.
	\]
	
	In the case where \( u \in \mathscr{D}'(U\times\mathbb{T}^m) \), for each \( \xi \in \mathbb{Z}^m \), its partial Fourier coefficient is the distribution \( \widehat{u}_\xi \in \mathscr{D}'(U) \) defined by
	\[
	\langle \widehat{u}_\xi, \phi \rangle = \langle u(t,x), \phi(t) \otimes e^{-i\xi\cdot x} \rangle, \quad \phi \in C_c^\infty(U).
	\]
	We now recall a characterization of smoothness in terms of Fourier coefficients.
	\begin{proposition}
	A distribution \( u \in \mathscr{D}'(U \times \mathbb{T}^m) \) belongs to \( C^\infty(U \times \mathbb{T}^m) \) if and only if:
	\begin{itemize}
		\item[(1)] \( \widehat{u}_\xi \in C^\infty(U) \) for all \( \xi \in \mathbb{Z}^m \);
		\item[(2)] For every compact set \( K \subset U \), every multi-index \( \alpha \in \mathbb{N}_0^n \), and every \( N \in \mathbb{N}_0 \), there exists a constant \( C > 0 \) such that
		\[
		\sup_{t \in K} \left| \partial_t^\alpha \widehat{u}_\xi(t) \right| \leq C \,(1 + |\xi|)^{-N}, \quad \xi \in \mathbb{Z}^m.
		\]
	\end{itemize}
	
	Under these conditions, \( u \) admits the Fourier expansion
	\[
	u(t,x) = \sum_{\xi \in \mathbb{Z}^m} \widehat{u}_\xi(t) \, e^{i\xi \cdot x},
	\]
	with uniform convergence on compact subsets of \( U \times \mathbb{T}^m \). 
	Moreover, \( u = 0 \) in \( \mathscr{D}'(U \times \mathbb{T}^m) \) if and only if \( \widehat{u}_\xi \equiv 0 \) for all \( \xi \in \mathbb{Z}^m \).
	\end{proposition}
		
	In the general case of \( M \times \mathbb{T}^m \), the partial Fourier coefficients are defined locally via charts. 
	This local definition is consistent with the global smooth structure of \( M \) and the standard Fourier analysis on \( \mathbb{T}^m \).

	\subsection{Differential forms on \(M \times \mathbb{T}^m\)}	\
		
	We now consider differential forms on \( M \times \mathbb{T}^m \) that involve only differentials coming from the manifold \( M \). We denote the space of such forms of degree \( q \) by \( \mathsf{\Lambda}^{0,q} C^\infty(M \times \mathbb{T}^m) \), where the bigrading \( (0,q) \) indicates degree 0 in the torus variables \( x \) and degree \( q \) in the manifold variables \( t \). In a local coordinate chart \( (U,t_1,\dots,t_n) \) of \( M \), a form \( f \in \mathsf{\Lambda}^{0,q} C^\infty(M \times \mathbb{T}^m) \) can be written as
	\[
	f|_{U\times\mathbb{T}^m} = \sum_{|J|=q} f_J(t,x) \,\mathrm{d}t_J,
	\]
	where \( J = (j_1,\dots,j_q) \) denotes an ordered multi-index of length \( q \), with \( 1 \leq j_1 < \dots < j_q \leq n \), and \( \mathrm{d}t_J = \mathrm{d}t_{j_1} \wedge \dots \wedge \mathrm{d}t_{j_q} \).
	
	For \( \xi \in \mathbb{Z}^m \), the Fourier coefficient of \( f \) in local coordinates is given by
	\[
	\widehat{f}_\xi(t) = \sum_{|J|=q} (\widehat{f_J})_\xi(t) \,\mathrm{d}t_J \ \in \mathsf{\Lambda}^q C^\infty(U),
	\]
	yielding the expansion
	\[
	f = \sum_{\xi \in \mathbb{Z}^m} \ \sum_{|J|=q} (\widehat{f_J})_\xi(t) \, e^{i\xi\cdot x} \,\mathrm{d}t_J,
	\]
	with uniform convergence on compact subsets. The smoothness of \( f \) is characterized by the same decay condition on \( (\widehat{f_J})_\xi \) as in the scalar case.
	
	The local definition of partial Fourier coefficients is invariant under coordinate changes (see, e.g., \cite[Section~5]{ADL2023gh}). Therefore, \( \widehat{f}_\xi \in \mathsf{\Lambda}^{0,q} C^\infty(M) \) is globally well defined, and
	\[
	f = \sum_{\xi \in \mathbb{Z}^m} \widehat{f}_\xi \, e^{i\xi\cdot x}
	\]
	holds with uniform convergence on compact subsets.
		
	This construction provides the natural framework for defining differential operators acting on $\mathsf{\Lambda}^{0,q} C^\infty(M \times \mathbb{T}^m)$ through their action on the coefficients $(\widehat{f_J})_\xi$.
	
	\begin{proposition}\label{cont_fourier}
		For each \( \xi \in \mathbb{Z}^m \), the map
		\begin{equation*}\label{f_fxi}
			f \in \mathsf{\Lambda}^{0,q} C^\infty(M \times \mathbb{T}^m) 
			\ \longmapsto\ 
			\widehat{f}_\xi \in \mathsf{\Lambda}^q C^\infty(M)
		\end{equation*}
		is continuous.
	\end{proposition}

	\begin{proof}
		To prove continuity, we must show that for any seminorm 
		$\|\cdot\|_{\ell,K}$ defining the topology of 
		$\mathsf{\Lambda}^q C^\infty(M)$, 
		there exist a seminorm $\|\cdot\|_{\ell',K'}$ 
		on $\mathsf{\Lambda}^{0,q} C^\infty(M \times \mathbb{T}^m)$ 
		and a constant $C > 0$ such that 
		\[
		\|\widehat{f}_\xi\|_{\ell,K} \leq C\,\|f\|_{\ell',K'}.
		\]
		
		Let $\|\cdot\|_{\ell,K}$ be a seminorm on $\mathsf{\Lambda}^q C^\infty(M)$. 
		In local coordinates, the coefficients of $\widehat{f}_\xi$ are given by
		\[
		(\widehat{f_J})_\xi(t) = \frac{1}{(2\pi)^m} 
		\int_{\mathbb{T}^m} e^{-i\xi\cdot x}\, f_J(t,x)\, \mathrm{d}x.
		\]
		Differentiating with respect to $t$ gives
		\[
		\partial_t^\alpha (\widehat{f_J})_\xi(t)
		= \frac{1}{(2\pi)^m} 
		\int_{\mathbb{T}^m} e^{-i\xi\cdot x}\,
		(\partial_t^\alpha f_J)(t,x)\, \mathrm{d}x.
		\]
		Taking the supremum over $t \in K_0$, where $K_0$ is a compact set in the chart domain, we obtain
		\begin{align*}
			\sup_{t \in K_0} 
			|\partial_t^\alpha (\widehat{f_J})_\xi(t)|
			&\leq \sup_{t \in K_0} 
			\frac{1}{(2\pi)^m} 
			\int_{\mathbb{T}^m} 
			|(\partial_t^\alpha f_J)(t,x)|\, \mathrm{d}x \\
			&\leq \frac{1}{(2\pi)^m} 
			\int_{\mathbb{T}^m} 
			\sup_{t \in K_0} |(\partial_t^\alpha f_J)(t,x)|\, \mathrm{d}x \\
			&\leq 
			\sup_{(t,x) \in K_0 \times \mathbb{T}^m} 
			|(\partial_t^\alpha f_J)(t,x)|.
		\end{align*}
		This shows that the seminorms for the coefficients of $\widehat{f}_\xi$ on $M$ 
		are bounded by the corresponding seminorms for the coefficients of $f$ 
		on $M \times \mathbb{T}^m$. 
		Summing over a finite atlas covering a compact set $K \subset M$, 
		we conclude that $\|\widehat{f}_\xi\|_{\ell,K}$ is bounded by a corresponding seminorm for $f$ 
		on $M \times \mathbb{T}^m$ (specifically, one defined on $K \times \mathbb{T}^m$). 
		This establishes that the map 
		$f \mapsto \widehat{f}_\xi$ 
		is continuous from 
		$\mathsf{\Lambda}^{0,q} C^\infty(M \times \mathbb{T}^m)$ 
		to $\mathsf{\Lambda}^q C^\infty(M)$.
	\end{proof}

	\subsection{Projective Limits of Locally Convex Spaces}\label{FS_spaces} \
	
	Let $\{E_k\}_{k\in\mathbb{N}_0}$ be a sequence of locally convex topological vector spaces equipped with linear maps $\phi_{j,k}:E_k\to E_j$ for $j\geq k$, such that each $\phi_{j,k}$ is compact and satisfies the compatibility condition $\phi_{j,k}\circ\phi_{k,\ell}=\phi_{j,\ell}$. The projective limit of this system is the subspace $E$ of the product $\prod_{k\in\mathbb{N}_0}E_k$ defined by
	\[
	E=\{(x_k)_{k\in\mathbb{N}_0} : x_k = \phi_{k,j}(x_j), \ j\geq k\},
	\]
	endowed with the coarsest locally convex topology for which each projection
	\[
	p_j:E\to E_j,\quad p_j(x)=x_j,\quad j\in\mathbb{N}_0,
	\]
	is continuous. In this case we write $E=\projlim E_k$.
	
	If each $\phi_{j,k}$ is injective, then so is every projection $p_j$. In particular, for $j\geq k$ we may identify $E_j$ as a vector subspace of $E_k$, and $E$ as a subspace of all $E_k$. In this setting, $E$ can also be viewed as the intersection $\bigcap_{k\in\mathbb{N}_0} E_k$, endowed with the coarsest locally convex topology that makes each inclusion $E\hookrightarrow E_k$ continuous.
	
	H.~Komatsu proved in \cite{Komatsu1967} that a topological vector space is a Fréchet–Schwartz (FS) space if and only if it is the projective limit of a compact sequence of locally convex spaces. Furthermore, by \cite[Lemma~2]{Komatsu1967}, any compact sequence of locally convex spaces is equivalent to a compact sequence of Banach spaces. Indeed, let $\phi:X\to Y$ be a compact map and let $V\subset X$ be an absolutely convex neighbourhood of $0$ such that $\phi(V)=A\subset Y$ is a precompact, absolutely convex subset of $Y$. Since $V$ is a neighbourhood of $0$, it generates $X$, and therefore $\mathrm{ran}(\phi)\subset Y_A:=\mathrm{span}(\overline{A})$. If $\tilde{\phi}$ denotes the restriction of $\phi$ with codomain $Y_A$, we can factor $\phi$ as
	\[
	X\overset{\tilde\phi}{\longrightarrow} Y_A \overset{\iota}{\longrightarrow} Y,
	\]
	where $\iota:Y_A\to Y$ is the inclusion map. By \cite[Chap.~III, p.~8]{Bourbaki}, the space $Y_A$ is Banach when endowed with the Minkowski functional $p_A$ of $A$,
	\[
	p_A(y)=\inf\{r\geq 0:\ y\in rA\}.
	\]
	Moreover, $\tilde{\phi}$ is injective whenever $\phi$ is injective.
	
	Applying this construction to the sequence $\{E_k\}$, and assuming that each $\phi_{j,k}$ is injective, we obtain
	\[
	\cdots \longrightarrow E_2\overset{\tilde\phi_{1,2}}{\longrightarrow} \tilde{E}_1 
	\overset{\iota}{\longrightarrow} E_1
	\overset{\tilde\phi_{0,1}}{\longrightarrow} \tilde{E}_0 
	\overset{\iota}{\longrightarrow} E_0,
	\]
	where each $\tilde{E}_k$ is a Banach space, the maps $\tilde{\phi}_{j,j+1}$ are injective, and the inclusions $\iota$ are compact. Setting $\psi_{j,j+1}=\iota\circ\tilde{\phi}_{j,j+1}$ yields a compact sequence of Banach spaces equivalent to the original sequence. In particular, if the $\phi_{j,j+1}$ are injective, then so are the $\psi_{j,j+1}$, and we may identify the projective limit $E$ as a subspace of each $E_k$.
	
	We now recall a fundamental result, due to G.~Araújo~\cite{Araujo2017}, which characterizes the closedness of the range for continuous linear maps between FS or DFS spaces.
	
	\begin{lemma}\label{lema_gabriel}
		Let $L:X\to Y$ be a linear operator between FS or DFS spaces. 
		The following conditions are equivalent:
		\begin{enumerate}
			\item $\mathrm{ran}(L)$ is closed in $Y$;
			\item $\mathrm{ran}(^tL)$ is closed in $X'$;
			\item $\mathrm{ran}(L)$ is sequentially closed in $Y$;
			\item $\mathrm{ran}(^tL)$ is sequentially closed in $X'$;
			\item $\mathrm{ran}(L)=\ker(^tL)^\perp$;
			\item $\mathrm{ran}(^tL)=\ker(L)^\perp$.
		\end{enumerate}
	\end{lemma}
	
	\begin{proof}
		See \cite[Lemma~2.2]{Araujo2017}.
	\end{proof}
	
	In the next subsection, we show that $C^\infty(M)$ is the projective limit of a compact sequence of nested Fréchet spaces, and therefore is an FS space.
	
	\subsection{Local Sobolev Spaces}\label{Loc_Sob} \
	
	Let \( U \subset \mathbb{R}^n \) be an open set. 
	For \( k \in \mathbb{N}_0 \), we denote by \( \mathscr{H}^k(U) \) the standard Sobolev space
	\[
	\mathscr{H}^k(U)
	= \big\{ u \in \mathscr{D}'(U) : 
	\partial^\alpha u \in L^2(U) \text{ for all } \alpha \in \mathbb{N}_0^n,\ |\alpha| \le k \big\},
	\]
	where derivatives are understood in the sense of distributions.
	
	Let \( \mathscr{H}_0^k(U) \) denote the closure of \( C_c^\infty(U) \) in \( \mathscr{H}^k(U) \). 
	For every \( \varphi \in C_c^\infty(U) \), the multiplication operator 
	\( u \mapsto \varphi u \) acts continuously from \( \mathscr{H}^k(U) \) into \( \mathscr{H}_0^k(U) \).
	
	The local Sobolev space of order \( k \in \mathbb{N}_0 \) is defined by
	\[
	\mathscr{H}^k_{\mathrm{loc}}(U)
	= \big\{ u \in \mathscr{D}'(U) : 
	\varphi u \in \mathscr{H}^k(U) \text{ for every } \varphi \in C_c^\infty(U) \big\}.
	\]
	
	\begin{remark}
		For the purposes of this work, it will be convenient to use the equivalent characterizations of local Sobolev spaces described by N.~Antoni{\'c} and K.~Burazin in~\cite{H_loc}. According to their results, the definition of 
		\( \mathscr{H}^k_{\mathrm{loc}}(U) \) can be restricted to a family of test functions  \( \Phi \subset C_c^\infty(U) \) satisfying the following property: 
		for every \( x \in U \), there exists \( \varphi \in \Phi \) such that \( \Re(\varphi)(x) > 0 \).
		By~\cite[Lemma~3]{H_loc}, this new definition of \( \mathscr{H}^k_{\mathrm{loc}}(U) \) is independent of the choice of such a family~\( \Phi \).
	\end{remark}
	
	We endow each \( \mathscr{H}^k_{\mathrm{loc}}(U) \) with the locally convex topology generated by the seminorms
	\[
	\|u\|_{k,\varphi} := \|\varphi u\|_{\mathscr{H}^k(U)}, 
	\qquad \varphi \in \Phi.
	\]
	These seminorms make \( \mathscr{H}^k_{\mathrm{loc}}(U) \) a Fréchet space~\cite[Theorem~1]{H_loc}.
	Moreover, the canonical embeddings 
	\( C^\infty(U) \hookrightarrow \mathscr{H}^k_{\mathrm{loc}}(U) \), 
	\( k \in \mathbb{N}_0 \), are continuous (see~\cite[Lemma~5.5]{Petersen}).
	
	Given a compact subset \( K \Subset U \), we also define
	\[
	\mathscr{H}_K^k(U)
	= \big\{ u \in \mathscr{H}^k(U) : \mathrm{supp}(u) \subset K \big\},
	\]
	which is a Banach subspace of \( \mathscr{H}^k(U) \).
	Then we set
	\[
	\mathscr{H}_c^k(U)
	= \bigcup_{K \Subset U} \mathscr{H}_K^k(U),
	\]
	endowed with the usual inductive limit topology.
	This makes \( \mathscr{H}_c^k(U) \) a locally convex (non-metrizable) topological vector space. 
	Moreover, the embeddings
	\[
	C_c^\infty(U) \hookrightarrow \mathscr{H}_c^k(U) \hookrightarrow \mathscr{H}^k(U)
	\]
	are continuous (see again~\cite[Lemma~5.2]{Petersen}).
	
	\begin{theorem}\label{comp_emb_Hloc}
		For each \( k \in \mathbb{N}_0 \), the embedding 
		\( \mathscr{H}_{\mathrm{loc}}^{k+1}(U) \hookrightarrow \mathscr{H}_{\mathrm{loc}}^k(U) \)
		is compact.
	\end{theorem}
	
	\begin{proof}
		See~\cite[Theorem~10]{H_loc}.
	\end{proof}
	
	\begin{remark}\label{cinfty_hloc}
		It is clear that 
		\( C^\infty(U) \subset \bigcap_{k\in\mathbb{N}_0} \mathscr{H}_{\mathrm{loc}}^k(U) \). 
		In fact, equality holds. 
		Indeed, let \( u \in \bigcap_{k\in\mathbb{N}_0} \mathscr{H}_{\mathrm{loc}}^k(U) \). 
		Then, for every bounded open subset \( V \subset U \), we have 
		\( u \in \mathscr{H}^k(V) \) for all \( k \in \mathbb{N}_0 \), 
		which implies \( u \in C^\infty(V) \). 
		Since smoothness is a local property, it follows that \( u \in C^\infty(U) \).
	\end{remark}
	
	By Theorem~\ref{comp_emb_Hloc}, if we endow \( C^\infty(U) \) with the projective limit topology induced by the sequence of compact embeddings 
	\(
	\mathscr{H}_{\mathrm{loc}}^{k+1}(U) \hookrightarrow \mathscr{H}_{\mathrm{loc}}^k(U)
	\),
	then \( C^\infty(U) \) becomes an FS space. 
	This topology is equivalent to the usual Fréchet topology of \( C^\infty(U) \).
		
	\begin{proposition}\label{int_Hloc}
		Consider \( C^\infty(U) \) endowed with its usual Fréchet topology, and let \( E \) be the Fréchet space obtained by endowing the vector space \( C^\infty(U) \) with the projective limit topology induced by the spaces \( \mathscr{H}_{\mathrm{loc}}^k(U) \). 
		Then the identity map \(I : C^\infty(U) \rightarrow E\) is a topological isomorphism.
	\end{proposition}
	
	\begin{proof}
		Let \( I : C^\infty(U) \to E \) be the identity map.  
		For each \( k \in \mathbb{N}_0 \), the canonical inclusion 
		\[
		C^\infty(U) \hookrightarrow \mathscr{H}_{\mathrm{loc}}^k(U)
		\]
		is continuous (see, as above,~\cite[Lemma~5.5]{Petersen}).  
		Since the projective topology of \(E\) is the coarsest locally convex topology making all these inclusions continuous, it follows that \( I \) itself is continuous.
		
		The map \( I \) is bijective by definition.  
		Moreover, both \( C^\infty(U) \) and \( E \) are Fréchet spaces.  
		By the Open Mapping Theorem, a bijective continuous linear map between Fréchet spaces has a continuous inverse.  
		Hence \( I^{-1} \) is continuous, and therefore \( I \) is a topological isomorphism.
	\end{proof}
	
	\begin{remark}
		This result shows that the projective limit construction of \( C^\infty(U) \) via local Sobolev spaces does not modify its topology. 
		It only provides an alternative characterization of the Fréchet–Schwartz structure of \( C^\infty(U) \) in terms of the compact sequence 
		\(
		\mathscr{H}_{\mathrm{loc}}^{k+1}(U) \hookrightarrow \mathscr{H}_{\mathrm{loc}}^k(U)
		\).
	\end{remark}
	
	Now, let \( M \) be an \(n\)-dimensional smooth paracompact manifold, and let 
	\(\mathcal{A} = \{(\Omega_j, \chi_j)\}_{j \in \mathbb{N}}\) 
	be a countable, locally finite atlas of \( M \) such that each \( \Omega_j \) is precompact.  
	Let \( \Phi = \{\varphi_j\}_{j \in \mathbb{N}} \subset C_c^\infty(M) \) be a partition of unity subordinate to the cover \( \{\Omega_j\}_{j \in \mathbb{N}} \).
	
	Following~\cite{Kumano-go}, we define, for each \( k \in \mathbb{N}_0 \),
	\[
	\mathscr{H}_{\mathrm{loc}}^k(M)
	:= 
	\Big\{
	u \in \mathscr{D}'(M)
	:\ 
	(\varphi_j u)\circ \chi_j^{-1} 
	\in 
	\mathscr{H}^k(\tilde{\Omega}_j),
	\quad 
	\forall j \in \mathbb{N}
	\Big\},
	\]
	where \( \tilde{\Omega}_j := \chi_j(\Omega_j) \subset \mathbb{R}^n \) is assumed to be bounded.
	
	Clearly, the family \( \Phi = \{\varphi_j\}_{j \in \mathbb{N}} \subset C_c^\infty(M) \) satisfies the property that, for every \( x \in M \), there exists \( \varphi_j \in \Phi \) such that \( \Re(\varphi_j)(x) > 0 \).  
	In this setting, each \( \mathscr{H}_{\mathrm{loc}}^k(M) \) becomes a locally convex topological vector space when endowed with the family of seminorms
	\[
	\|u\|_{k,j} 
	:= 
	\| (\varphi_j u)\circ \chi_j^{-1} \|_{\mathscr{H}^k(\tilde{\Omega}_j)},
	\qquad 
	j \in \mathbb{N}.
	\]
	
	This topology does not depend on the particular choice of the atlas \( \mathcal{A} \) or of the subordinate partition of unity \( \Phi \), 
	and it makes \( \mathscr{H}_{\mathrm{loc}}^k(M) \) a Fréchet space.  
	Furthermore, the canonical embeddings
	\[
	C^\infty(M) \hookrightarrow \mathscr{H}_{\mathrm{loc}}^k(M), 
	\qquad 
	k \in \mathbb{N}_0,
	\]
	are continuous.
	
	\begin{proposition}
		For each \( k \in \mathbb{N}_0 \), the space \( \mathscr{H}_{\mathrm{loc}}^k(M) \) is a Fréchet space.
	\end{proposition}
	
	\begin{proof}
		The topology of \( \mathscr{H}_{\mathrm{loc}}^k(M) \) is defined by the countable family of seminorms
		\[
		\|u\|_{k,j} = \|(\varphi_j u)\circ \chi_j^{-1}\|_{\mathscr{H}^k(\tilde{\Omega}_j)},
		\qquad j \in \mathbb{N},
		\]
		hence it is metrizable.  
		For instance, one may consider the metric
		\[
		d_k(u,v) 
		= 
		\sum_{j \in \mathbb{N}} 
		\frac{1}{2^j} 
		\frac{\|u - v\|_{k,j}}{1 + \|u - v\|_{k,j}}.
		\]
		We now show that \( (\mathscr{H}_{\mathrm{loc}}^k(M), d_k) \) is complete.
		
		Let \( \{u_n\}_{n\in\mathbb{N}} \) be a Cauchy sequence in \( \mathscr{H}_{\mathrm{loc}}^k(M) \).  
		Then, for each fixed \( j \in \mathbb{N} \), the sequence 
		\[
		\{(\varphi_j u_n)\circ \chi_j^{-1}\}_{n\in\mathbb{N}}
		\]
		is Cauchy in the Hilbert space \( \mathscr{H}^k(\tilde{\Omega}_j) \).  
		Since \( \mathscr{H}^k(\tilde{\Omega}_j) \) is complete, there exists 
		\(
		\tilde{u}_j \in \mathscr{H}^k(\tilde{\Omega}_j)
		\)
		such that
		\[
		(\varphi_j u_n)\circ \chi_j^{-1} \longrightarrow \tilde{u}_j
		\quad \text{in } \mathscr{H}^k(\tilde{\Omega}_j), \quad n \to \infty.
		\]
		
		We claim that there exists \( u \in \mathscr{H}_{\mathrm{loc}}^k(M) \) satisfying 
		\( (\varphi_j u)\circ \chi_j^{-1} = \tilde{u}_j \) for all \( j \).
		To define \( u \), note that for every \( x \in M \) there exists at least one \( j \) such that \( \varphi_j(x) \neq 0 \).
		On the set \( \{x : \varphi_j(x) \neq 0\} \), we may define
		\[
		u(x) := \frac{(\tilde{u}_j \circ \chi_j)(x)}{\varphi_j(x)}.
		\]
		We must check that this definition is independent of \(j\).  
		If \( \varphi_i(x) \neq 0 \) and \( \varphi_j(x) \neq 0 \), then on the overlap 
		\( \Omega_i \cap \Omega_j \) we have, for any $n$,
		\[
		\varphi_i(\varphi_j u_n)
		= 
		\varphi_j(\varphi_i u_n).
		\]
		Passing to the limit in \( \mathscr{H}^k(\tilde{\Omega}_i \cap \tilde{\Omega}_j) \)
		yields
		\(
		\varphi_i (\tilde{u}_j \circ \chi_j) = \varphi_j (U_i \circ \chi_i),
		\)
		which guarantees that the function \(u\) above is well defined, independently of the chart chosen.
		
		By construction, \( (\varphi_j u)\circ \chi_j^{-1} = \tilde{u}_j \in \mathscr{H}^k(\tilde{\Omega}_j) \) for all \(j\),
		so \( u \in \mathscr{H}_{\mathrm{loc}}^k(M) \).
		Finally,
		\[
		\|u_n - u\|_{k,j} 
		= 
		\|(\varphi_j(u_n - u))\circ \chi_j^{-1}\|_{\mathscr{H}^k(\tilde{\Omega}_j)} 
		= 
		\|(\varphi_j u_n)\circ \chi_j^{-1} - \tilde{u}_j\|_{\mathscr{H}^k(\tilde{\Omega}_j)} 
		\longrightarrow 0,
		\]
		as \( n \to \infty \) for each \( j \in \mathbb{N} \).
		Hence \( u_n \to u \) in \( \mathscr{H}_{\mathrm{loc}}^k(M) \), showing that the space is complete.
	\end{proof}

	\begin{proposition}\label{cont_emb_M}
		For each \( k \in \mathbb{N}_0 \), the embedding 
		\( C^\infty(M) \hookrightarrow \mathscr{H}_{\mathrm{loc}}^k(M) \) is continuous.
	\end{proposition}
	
	\begin{proof}
		Let \( \Phi = \{\varphi_j\}_{j \in \mathbb{N}} \subset C_c^\infty(M) \) 
		be the partition of unity and \( \mathcal{A} = \{(\Omega_j, \chi_j)\} \) the atlas defining the seminorms
		\[
		\|u\|_{k,j} = \|(\varphi_j u)\circ \chi_j^{-1}\|_{\mathscr{H}^k(\tilde{\Omega}_j)}.
		\]
		For each fixed \(j\), the operator \(u \mapsto \varphi_j u\) is continuous 
		from \(C^\infty(M)\) to \(C_c^\infty(\Omega_j)\), 
		and the inclusions \(C_c^\infty(\tilde{\Omega}_j) \hookrightarrow \mathscr{H}^k(\tilde{\Omega}_j)\),
		as well as the compositions with \(\chi_j^{-1}\), are continuous.  
		Hence, each seminorm \(\|\cdot\|_{k,j}\) is continuous on \(C^\infty(M)\), 
		which proves the continuity of the embedding.
	\end{proof}
	
	\begin{theorem}\label{comp_emb_Hloc_M}
		For each \( k \in \mathbb{N}_0 \), the embedding \(\mathscr{H}_{\mathrm{loc}}^{k+1}(M) \hookrightarrow \mathscr{H}_{\mathrm{loc}}^{k}(M)\) is compact.
	\end{theorem}
	
	\begin{proof}
		Let \( \{u_n\} \) be a bounded sequence in \( \mathscr{H}_{\mathrm{loc}}^{k+1}(M) \).
		We show that it admits a convergent subsequence in \( \mathscr{H}_{\mathrm{loc}}^{k}(M) \).
		
		Fix \( j \in \mathbb{N} \).
		By the continuity of multiplication by \( \varphi_j \) and of composition with \( \chi_j^{-1} \),
		the sequence \( \{(\varphi_j u_n)\circ\chi_j^{-1}\} \) is bounded in \( \mathscr{H}_0^{k+1}(\tilde{\Omega}_j) \).
		By the Rellich–Kondrachov theorem (see, e.g., \cite[Theorem~7.1]{Wloka}),
		the inclusion
		\(
		\mathscr{H}_0^{k+1}(\tilde{\Omega}_j) \hookrightarrow \mathscr{H}_0^{k}(\tilde{\Omega}_j)
		\)
		is compact.
		Hence, there exists a subsequence, denoted again by \( \{u_n\} \), such that
		\[
		(\varphi_j u_n)\circ\chi_j^{-1} \to \tilde{u}_j 
		\quad \text{in } \mathscr{H}^{k}(\tilde{\Omega}_j),
		\]
		for some \( \tilde{u}_j \in \mathscr{H}^{k}(\tilde{\Omega}_j) \).
		
		As in the proof of Proposition~\ref{cont_emb_M}, 
		the family \( \{\tilde{u}_j\} \) defines a well-defined \( u \in \mathscr{H}_{\mathrm{loc}}^k(M) \),
		by the local relation \( (\varphi_j u)\circ\chi_j^{-1} = \tilde{u}_j \).
		Finally,
		\[
		\|u_n - u\|_{k,j}
		= 
		\|(\varphi_j u_n)\circ\chi_j^{-1} - \tilde{u}_j\|_{\mathscr{H}^k(\tilde{\Omega}_j)}
		\longrightarrow 0
		\quad \text{for all } j,
		\]
		so \( u_n \to u \) in \( \mathscr{H}_{\mathrm{loc}}^k(M) \).
		Therefore, the embedding is compact.
	\end{proof}
	
	\begin{theorem}\label{int_Hloc_M}
		We have
		\[
		C^\infty(M)
		=
		\bigcap_{k\in\mathbb{N}_0} \mathscr{H}_{\mathrm{loc}}^k(M),
		\]
		and the usual Fréchet topology on \( C^\infty(M) \) is equivalent to the projective limit topology induced by the sequence 
		\(\{\mathscr{H}_{\mathrm{loc}}^k(M)\}_{k\in\mathbb{N}_0}\).
	\end{theorem}
	
	\begin{proof}
		The proof is analogous to that of Proposition~\ref{int_Hloc}, 
		using Proposition~\ref{cont_emb_M}, 
		Theorem~\ref{comp_emb_Hloc_M}, 
		and the Open Mapping Theorem.
	\end{proof}
	
	Hence, \( C^\infty(M) \) can be regarded as the projective limit of a nested sequence of Fréchet spaces with compact and injective inclusions. 
	In particular, \( C^\infty(M) \) is a Fréchet–Schwartz (FS) space, and consequently its strong dual \( \mathcal{E}'(M) \) is a DFS space.
	
	Furthermore, by the construction discussed at the end of the previous subsection, 
	we can obtain a nested sequence of Banach spaces with compact and injective inclusions whose projective limit is topologically isomorphic to \( C^\infty(M) \).
	The same argument applies to spaces of smooth sections of vector bundles over a smooth paracompact manifold \( M \).
	
	In particular, if \( \mathbb{T}^m = \mathbb{R}^m / 2\pi\mathbb{Z}^m \) is the \(m\)-dimensional torus, 
	then, for each \( q = 0, 1, \dots, n \),
	\[
	\mathsf{\Lambda}^q C^\infty(M) \text{ and } 
	\mathsf{\Lambda}^q C^\infty(M \times \mathbb{T}^m)  \text{ are FS spaces.}
	\]

	\section{Global regularity and matrices of cycles}\label{sec:glregcycles}
	
	In this section, we briefly recall some results from~\cite{CKMT} 
	concerning the global hypoellipticity of the operator
	\[
	\mathbb{L}^0 : C^\infty(M\times\mathbb{T}^m) \longrightarrow 
	\mathsf{\Lambda}^{0,1} C^\infty(M\times\mathbb{T}^m),
	\]
	defined by
	\begin{equation}\label{L_0}
		\mathbb{L}^0 u = \mathrm{d}_t u + \sum_{k=1}^m \omega_k \wedge \partial_{x_k}u,
	\end{equation}
	where \( \mathrm{d}_t \) denotes the partial exterior derivative on \(M\),
	and \( \omega_1, \dots, \omega_m \) are real-valued, closed \(1\)-forms on \(M\).
	Recall that \( \mathbb{L}^0 \) is said to be \emph{globally hypoelliptic} if
	\[
	u \in \mathscr{D}'(M\times\mathbb{T}^m), \quad
	\mathbb{L}^0 u \in \mathsf{\Lambda}^{0,1} C^\infty(M\times\mathbb{T}^m)
	\ \Longrightarrow\ 
	u \in C^\infty(M\times\mathbb{T}^m).
	\]
	
	Let \( \overline{M} \) be a smooth compact manifold with boundary,
	and denote by \( M \) its interior. 
	A \emph{scattering metric} on \( M \) is a Riemannian metric \( g \) 
	which, in a collar neighborhood of the boundary,
	takes the form
	\[
	g = \frac{\mathrm{d}\varrho^2}{\varrho^4} + \frac{g'}{\varrho^2},
	\]
	where \( \varrho : \overline{M} \to [0,+\infty) \) 
	is a boundary defining function 
	and \( g' \) is a smooth symmetric \(2\)-tensor 
	that restricts to a Riemannian metric on \( \partial M \).
	In this case, we say that \( (\overline{M}, g) \) is a \emph{scattering manifold}.
	This class is quite broad: every manifold that is the interior of a compact manifold with boundary admits such metrics.
	An important subclass of scattering manifolds consists of manifolds with cylindrical ends 
	(see the Appendix of~\cite{CD2021}).
	For further details on the theory of scattering manifolds, we refer to 
	\cite{Melrose_APS, Melrose_SST, Melrose_GST}.
	
	Throughout this section, we assume that \( M \) is a connected, orientable, 
	\( n \)-dimensional manifold which is the interior of a scattering manifold.
	Let \( H_{\mathrm{dR}}^1(\overline{M}, \partial M) \) denote the first relative de~Rham cohomology group of compactly supported smooth forms on \(M\),
	and define \(H_{\partial M}^1(M)\) as the image of the map
	\[
	H_{\mathrm{dR}}^1(\overline{M}, \partial M) \longrightarrow H_{\mathrm{dR}}^1(\overline{M}), 
	\ [\omega] \mapsto [\omega].
	\]
	
	We may regard \( H_{\partial M}^1(M) \) as a subspace of \( H_{\mathrm{dR}}^1(M) \): indeed, since \( H_{\mathrm{dR}}^1(\overline{M}) \) 	is the de~Rham cohomology of forms smooth up to the boundary, 
	its elements are determined by their restrictions to the interior \(M\), 
	so that \( H_{\mathrm{dR}}^1(\overline{M}) \subset H_{\mathrm{dR}}^1(M) \) naturally.
	
	We denote by \( \mathsf{\Lambda}^1 C_{\partial M}^\infty(M) \) 	the space of smooth closed \(1\)-forms \( \omega \) on \(M\) 	whose cohomology class satisfies \( [\omega] \in H_{\partial M}^1(M) \).
	
	\begin{definition}
		Let \( \omega \in \mathsf{\Lambda}^1 C^\infty(M) \) be a closed \(1\)-form on \(M\).
		We say that \( \omega \) is \emph{integral} if
		\[
		\frac{1}{2\pi} \int_\sigma \omega \in \mathbb{Z}
		\quad \text{for every smooth \(1\)-cycle } \sigma \text{ in } M.
		\]
	\end{definition}
	
	\begin{proposition}\label{uni_cov}
		Let \( \Pi : \hat{M} \to M \) be the universal covering of \(M\),
		and let \( \omega \) be a real-valued closed \(1\)-form on \(M\).
		Then the following conditions are equivalent:
		\begin{enumerate}
			\item[\textnormal{(i)}] \( \omega \) is integral;
			\item[\textnormal{(ii)}] 
			For every \( \psi \in C^\infty(\hat{M}) \) satisfying
			\(\mathrm{d}\psi = \Pi^* \omega\), we have
			\[
			P, Q \in \hat{M}, \ \Pi(P) = \Pi(Q)
			\ \Longrightarrow\ 
			\psi(P) - \psi(Q) \in 2\pi \mathbb{Z}.
			\]
		\end{enumerate}
	\end{proposition}
	
	\begin{proof}
		See~\cite[Proposition~3.4]{CKMT}.
	\end{proof}

	\begin{definition}
		Let \( \omega_1, \dots, \omega_m \in \mathsf{\Lambda}^1 C_{\partial M}^\infty(M) \)
		be a family of real-valued closed \(1\)-forms on \(M\).
		We say that the family \( \boldsymbol{\omega} = (\omega_1, \dots, \omega_m) \) is:
		\begin{enumerate}
			\item \emph{rational} if there exists \( \xi \in \mathbb{Z}^m \setminus \{0\} \) such that
			\[
			\xi \cdot \boldsymbol{\omega} := \sum_{k=1}^{m} \xi_k \omega_k
			\]
			is an integral \(1\)-form;
			
			\item \emph{Liouville} if \( \boldsymbol{\omega} \) is not rational and there exist
			a sequence of closed integral \(1\)-forms \( \{\omega_j\} \)
			and a sequence \( \{\xi^{(j)}\} \subset \mathbb{Z}^m \) with \( |\xi^{(j)}| \to \infty \),
			such that the family
			\[
			\bigl\{\, |\xi^{(j)}|^j\,(\xi^{(j)} \cdot \boldsymbol{\omega} - \omega_j) \,\bigr\}
			\]
			is bounded in \( \mathsf{\Lambda}^1 C^\infty(M) \).
		\end{enumerate}
	\end{definition}
	
	Let 
	\(
	[\vartheta_1], \dots, [\vartheta_d]
	\)
	be a basis of \( H_{\partial M}^1(M) \),
	and let 
	\(
	\sigma_1, \dots, \sigma_d
	\)
	be smooth loops whose homology classes form a basis of a subspace of 
	\( H_1(M; \mathbb{R}) \)
	dual to \( H_{\partial M}^1(M) \) in the sense that
	\[
	\int_{\sigma_\ell} \vartheta_k = \delta_{\ell k},
	\qquad k,\ell = 1,\dots,d.
	\]
	
	Given a family 
	\( \boldsymbol{\omega} = (\omega_1, \dots, \omega_m) \),
	we associate to it the \emph{matrix of cycles}
	\( A(\boldsymbol{\omega}) \in M_{d\times m}(\mathbb{R}) \),
	whose entries are defined by
	\[
	A(\boldsymbol{\omega})_{\ell k}
	= 
	\frac{1}{2\pi}
	\int_{\sigma_\ell} \omega_k.
	\]
	Then, for every \( \xi \in \mathbb{Z}^m \),
	\[
	A(\boldsymbol{\omega})\,\xi
	=
	\frac{1}{2\pi}
	\left(
	\int_{\sigma_1} \xi\!\cdot\!\boldsymbol{\omega},
	\dots,
	\int_{\sigma_d} \xi\!\cdot\!\boldsymbol{\omega}
	\right)^{\!\top}.
	\]
	
	Notice that the definition of \( A(\boldsymbol{\omega}) \)
	depends only on the cohomology classes
	\(
	[\omega_1], \dots, [\omega_m] \in H_{\partial M}^1(M).
	\)
	Hence, we obtain a well-defined linear map
	\[
	A : H_{\partial M}^1(M)^m \longrightarrow M_{d\times m}(\mathbb{R}),
	\qquad
	\boldsymbol{\omega} \longmapsto A(\boldsymbol{\omega}).
	\]
	
	\begin{definition}
		A matrix \( \boldsymbol{A} \in M_{d\times m}(\mathbb{R}) \) is said to satisfy the 
		\emph{Diophantine condition} \((\mathrm{DC})\)
		if there exist constants \( C, \rho > 0 \) such that
		\begin{equation}\label{DC}
			|\eta + \boldsymbol{A}\xi| \geq C\, (|\eta| + |\xi|)^{-\rho},
			\tag{DC}
		\end{equation}
		for all \( (\xi, \eta) \in \mathbb{Z}^m \times \mathbb{Z}^d \setminus \{(0,0)\} \).
	\end{definition}
	
	\begin{lemma}\label{lemma_dioph}
		If a matrix \( \boldsymbol{A} \in M_{d\times m}(\mathbb{R}) \) satisfies \eqref{DC}, 
		then there exist constants \( C > 0 \) and \( N \in \mathbb{N} \) such that
		\[
		\max_{1 \leq \ell \leq d} 
		\big|\, 1 - e^{2\pi i\, \xi \cdot a_\ell}\, \big|
		\geq
		C\, (1 + |\xi|)^{-N},
		\quad 
		\forall\, \xi \in \mathbb{Z}^m \setminus \{0\},
		\]
		where \( a_\ell \in \mathbb{R}^m \) denotes the \(\ell\)-th row of \( \boldsymbol{A} \).
	\end{lemma}
	
	\begin{proof}
		See~\cite[Lemma~3.7]{CKMT}.
	\end{proof}
	
	\begin{proposition}\label{prop_rat_Liouville_matrix}
		Let \( \boldsymbol{\omega} = (\omega_1, \dots, \omega_m) \) 
		be a family of closed \(1\)-forms as above.
		Then:
		\begin{itemize}
			\item[(i)] \( \boldsymbol{\omega} \) is \emph{rational} 
			if and only if 
			\[
			A(\boldsymbol{\omega})(\mathbb{Z}^m \setminus \{0\}) 
			\cap \mathbb{Z}^d \neq \varnothing;
			\]
			\item[(ii)] \( \boldsymbol{\omega} \) is \emph{Liouville} 
			if and only if 
			it is not rational and \( A(\boldsymbol{\omega}) \) does not satisfy \eqref{DC}.
		\end{itemize}
	\end{proposition}
	
	\begin{proof}
		See~\cite[Proposition~3.8]{CKMT}.
	\end{proof}
	
	Hence, the global regularity of \( \mathbb{L}^0 \) can be characterized as in the subsequent Theorem \ref{thm_global_hypo}.
	
	\begin{theorem}\label{thm_global_hypo}
		Let 
		\( \omega_1, \dots, \omega_m \in \mathsf{\Lambda}^1 C^\infty_{\partial M}(M) \)
		be a family of real-valued closed \(1\)-forms on \(M\).
		Then the operator
		\[
		\mathbb{L}^0 = \mathrm{d}_t + \sum_{k=1}^m \omega_k \wedge \partial_{x_k}
		\]
		is globally hypoelliptic 
		if and only if the family 
		\( \boldsymbol{\omega} = (\omega_1, \dots, \omega_m) \)
		is neither rational nor Liouville.
	\end{theorem}
	
	\begin{proof}
		See~\cite[Theorem~3.3]{CKMT}.
	\end{proof}

	\section{Regularity and closedness of the range in an abstract setting}\label{sec:regclrng}
	
	In this section, we obtain conditions for the closedness of the range of operators on FS spaces under suitable assumptions. This characterization is based on an abstract notion of global hypoellipticity on pairs of topological vector spaces, introduced in~\cite{AFR22proc}.  As a consequence, we show that a weaker notion of global regularity implies the closedness of the range of differential operators acting on smooth sections of vector bundles over a general smooth paracompact manifold.	
	
	First, we recall some basic notions and results.  A pair of topological vector spaces is a 2-tuple 	\((E^\sharp, E)\) where \(E\) is a linear subspace of \(E^\sharp\),
	endowed with a topology finer than the one induced by \(E^\sharp\).  
		
	\begin{definition}
		Given two pairs $(E^\sharp,E),(F^\sharp,F)$ of topological vector spaces, we say that $P:(E^\sharp,E)\to(F^\sharp,F)$ is a map of pairs if $P:E^\sharp\to F^\sharp$ is continuous, $P(E)\subset F$ and the induced map $P:E\to F$ is continuous. In this case, we say that:
	\begin{itemize}
		\item \(P\) satisfies property \((\mathcal{H})\) if
		\(
		u\in E^\sharp,\ P u \in F \ \Rightarrow\  u \in E;
		\)
		\item \(P\) satisfies property \((\mathcal{H}')\) if
		\(
		u\in E^\sharp,\ P u \in F \ \Rightarrow\ 
		\exists v\in E \text{ such that } P u = P v.
		\)
	\end{itemize}
	\end{definition}

	Note that, if \(P:(E^\sharp,E)\to(F^\sharp,F)\) is a map of pairs,  then \((E^\sharp/\ker P,\ E/(E\cap\ker P))\) is again a pair of topological vector spaces, and \(P\) descends to a quotient map of pairs
	\[
	P' : (E^\sharp/\ker P,\ E/(E\cap\ker P)) \longrightarrow (F^\sharp,F).
	\]
	
	\begin{proposition}\label{prop_pairs}
		Let \(P:(E^\sharp,E)\to(F^\sharp,F)\) be a map of pairs with \(F^\sharp\) Hausdorff. 
		If \(P\) satisfies property \((\mathcal{H})\), 
		then the graph of \(P:E\to F\) is closed in \(E^\sharp\times F\), that is, the map
		\[
		\gamma_P : E \longrightarrow E^\sharp\times F,\quad u\mapsto (u,Pu),
		\]
		has closed range.
	\end{proposition}
	
	\begin{proof}
		See~\cite[Lemma~5.2]{AFR22proc}.
	\end{proof}
	
	\begin{lemma}\label{frechet_closed}
		A bounded linear map \(P:E\to F\) between Fréchet spaces has closed range if and only if the following holds:
		whenever \(\{u_\ell\}\subset E\) satisfies \(P u_\ell \to 0\), there exists a sequence \(\{v_\ell\}\subset E\) such that \(P u_\ell = P v_\ell\) for all \(\ell\in\mathbb{N}\) and \(v_\ell \to 0\).
	\end{lemma}
	
	\begin{proof}
		See~\cite[p.~18]{kothe_TVS}.
	\end{proof}
	
	Inspired by the proof of \cite[Theorem 3.5]{AFR2022}, we obtain the following main result of this section.
	
	\begin{theorem}\label{FS_GHGS}
		Let \(E,F\) be FS spaces represented as projective limits of nested compact sequences of Banach spaces,
		\[
		E = \projlim E_k, \qquad F = \projlim F_k.
		\]
		Let \(P:(E_0,E)\to(F^\sharp,F)\) be a map of pairs, with \(F^\sharp\) Hausdorff.
		If \(P\) satisfies property \((\mathcal{H}')\), then the range of \(P:E\to F\) is closed.
	\end{theorem}
	
	\begin{proof}
		Firstly, we observe that for any map of pairs \(P:(E^\sharp,E)\to(F^\sharp,F)\), the quotient \((E^\sharp/\ker P,\ E/(E\cap\ker P))\) is again a pair of topological vector spaces, and \(P\) descends to a quotient map of pairs
		\[
		P' : (E^\sharp/\ker P,\ E/(E\cap\ker P)) \longrightarrow (F^\sharp,F).
		\]
		Since \(P\) satisfies the hypotheses of Proposition~\ref{prop_pairs},
		we apply this to our setting with \(K_0=\ker P\) and \(K=K_0\cap E\),
		so that \(P'\) satisfies property \((\mathcal{H})\).

		By Proposition~\ref{prop_pairs}, the associated graph map 
		\[
		\gamma_{P'} : E/K \longrightarrow (E_0/K_0)\times F, \quad u+K \longmapsto (u+K_0,Pu),
		\]
		has closed range.

		Notice that \(Y=\mathrm{ran}(\gamma_{P'})\) is a closed subspace of the Fréchet space \((E_0/K_0)\times F\), hence \(Y\) is Fréchet. Since \(E/K\) is FS, the Open Mapping Theorem implies that \(\gamma_{P'}^{-1}:Y\to E/K\) is continuous. Consequently, for each \(k\in\mathbb{N}\) there exist \(j\in\mathbb{N}_0\) and \(C>0\) such that
		\[
		\|u+K\|_{E_k/K} \le C\big(\|u+K_0\|_{E_0/K_0} + \|Pu\|_{F_j}\big), \quad u\in E.
		\]

		We claim that there exists \(C'>0\) satisfying
		\begin{equation}\label{ineq_T}
			\|u+K\|_{E_k/K} \le C'\|Pu\|_{F_j},\quad u\in E.
		\end{equation}
		
		Assume, by contradiction, that \eqref{ineq_T} fails. Then for each \(\ell\in\mathbb{N}\) there exists \(u_\ell\in E\) such that \(\|u_\ell+K\|_{E_k/K} > \ell\,\|Pu_\ell\|_{F_j}\). By rescaling, we may assume \(\|u_\ell+K\|_{E_k/K}=1\).
		The quotient norm is given by
		\[
		\|u+K\|_{E_k/K} = 
		\inf\big\{\,\|v\|_{E_k} : v\in E_k,\ u-v\in K\,\big\}.
		\]
		Thus, for each \(\ell\) we can find \(v_\ell\in E_k\) with \(u_\ell-v_\ell\in K\) and \(\|v_\ell\|_{E_k}\le2\).
		The sequence \(\{v_\ell\}\) is bounded in \(E_k\), and the compact inclusion \(E_k\hookrightarrow E_0\) 
		implies that a subsequence \(\{v_{\ell'}\}\) converges in \(E_0\) to some \(v\in E_0\).

		Since \(P v_{\ell'}\to0\), continuity gives \(v\in K_0\). Therefore, \(u_{\ell'}+K_0=v_{\ell'}+K_0\to v+K_0 = 0+K_0\), and hence
		\[
		1 = \|u_{\ell'}+K\|_{E_k/K} \le C\big(\|u_{\ell'}+K_0\|_{E_0/K_0}+\|Pu_{\ell'}\|_{F_j}\big)\to0,
		\]
		a contradiction. We conclude that \eqref{ineq_T} holds true.  By Lemma~\ref{frechet_closed}, the range of the quotient map 
		\(u+K\mapsto Pu\) is closed, and so is the range of \(P:E\to F\).
	\end{proof}
	
	\begin{definition}
		Let \( \mathbb{E} \) and \( \mathbb{F} \) be smooth complex (or real) vector bundles over a smooth paracompact manifold \( M \). A differential operator
		\[
		P : C^\infty(M,\mathbb{E}) \longrightarrow C^\infty(M,\mathbb{F})
		\]
		is said to be \emph{almost globally hypoelliptic} if
		\[
		u \in \mathscr{D}'(M,\mathbb{E}), \ P u \in C^\infty(M,\mathbb{F}) \ \Rightarrow \ \exists\, v \in C^\infty(M,\mathbb{E}) \text{ such that } P u = P v.
		\]
	\end{definition}
	
	\begin{corollary}\label{AGH_closed}
		Let \(\mathbb{E}\) and \(\mathbb{F}\) be smooth vector bundles over a paracompact manifold \(M\).
		If a differential operator 
		\(P:C^\infty(M,\mathbb{E})\to C^\infty(M,\mathbb{F})\)
		is almost globally hypoelliptic, then \(P\) has closed range.
	\end{corollary}
	\begin{proof}
		By the results in Sections~\ref{FS_spaces} and~\ref{Loc_Sob},
		the spaces \(C^\infty(M,\mathbb{E})\) and \(C^\infty(M,\mathbb{F})\)
		are FS spaces, each realized as a projective limit of a nested sequence
		of Banach spaces with compact inclusions,
		\(\{E_k\}_{k\in\mathbb{N}_0}\) and \(\{F_k\}_{k\in\mathbb{N}_0}\), respectively.
		In particular, we may view \(C^\infty(M,\mathbb{E})\subset E_0\) and
		\(C^\infty(M,\mathbb{F})\subset F_0\).
		
		Now, the almost global hypoellipticity of \(P\) means precisely that
		the map
		\[
		P:(E_0,\, C^\infty(M,\mathbb{E})) \longrightarrow
		(\mathscr{D}'(M,\mathbb{F}),\, C^\infty(M,\mathbb{F}))
		\]
		is a map of pairs satisfying property \((\mathcal{H}')\),
		and \(\mathscr{D}'(M,\mathbb{F})\) is Hausdorff.
		Hence, all the hypotheses of Theorem~\ref{FS_GHGS} are satisfied,
		and the closedness of the range of \(P\) follows.
	\end{proof}


\begin{remark}
	In particular, \cite[Theorem~3.5]{AFR2022} (see also \cite[Theorem~2.2]{AFJR2024})
	remains valid in the non-compact setting considered here.
	The converse, however, does not hold in general:
	see \cite[Example~3.6]{AFR2022} for a vector field on the circle
	whose range is closed but which is not almost globally hypoelliptic.
	Fortunately, the converse is true for the class of operators that will concern us,
	as shown in the subsequent Theorem \ref{gs_agh_equiv}. 
\end{remark}

\begin{theorem}\label{gs_agh_equiv}
	The operator \(	\mathbb{L}^0 : C^\infty(M\times\mathbb{T}^m) \longrightarrow \mathsf{\Lambda}^{0,1} C^\infty(M\times\mathbb{T}^m)\), defined by
	\begin{equation}\label{L_0bis}
		\mathbb{L}^0 u = \mathrm{d}_t u + \sum_{k=1}^m \omega_k \wedge \partial_{x_k}u,
	\end{equation}
	has closed range if and only if it is almost globally hypoelliptic.
\end{theorem}

\begin{proof}
	If \( \mathbb{L}^0 \) is almost globally hypoelliptic,
	then it has closed range, by Corollary~\ref{AGH_closed}.
	
	Conversely, suppose that \( \mathbb{L}^0 \) has closed range,
	and let \( u \in \mathscr{D}'(M\times\mathbb{T}^m) \) be a distribution that satisfies
	\[
	\mathbb{L}^0 u = f \in \mathsf{\Lambda}^{0,1}C^\infty(M\times\mathbb{T}^m).
	\]
	Taking the partial Fourier expansion in the torus variables, we obtain, for each
	\( \xi \in \mathbb{Z}^m \),
	\[
	\widehat{f}_\xi(t) 	= \mathbb{L}_\xi^0 \widehat{u}_\xi(t) = \mathrm{d}_t \widehat{u}_\xi(t)
	+ i(\xi\!\cdot\!\boldsymbol{\omega})\,\widehat{u}_\xi(t).
	\]
	
	By \cite[Corollary~2.4]{CKMT}, each Fourier coefficient \( \widehat{u}_\xi \) belongs to \( C^\infty(M) \).
	Hence, we can reconstruct \( f \) as
	\[
	f = \lim_{\ell\to\infty} 	\sum_{|\xi|\le\ell} e^{i\xi\cdot x}\, 	\mathbb{L}_\xi^0 \widehat{u}_\xi(t)
	= \lim_{\ell\to\infty} 	\mathbb{L}^0\!\left( \sum_{|\xi|\le\ell} e^{i\xi\cdot x}\widehat{u}_\xi(t) \right),
	\]
	with convergence in \( \mathsf{\Lambda}^{0,1}C^\infty(M\times\mathbb{T}^m) \).
	
	Therefore, \( f \) lies in the closure of \( \mathrm{ran}(\mathbb{L}^0) \). Since the range of \( \mathbb{L}^0 \) is closed, there exists \( v \in C^\infty(M\times\mathbb{T}^m) \) such that \( \mathbb{L}^0 v = f \). This proves that \( \mathbb{L}^0 \) is almost globally hypoelliptic.
\end{proof}

	\section{Global solvability}\label{sec:globsol}
 
	 In this section, we establish necessary and sufficient conditions for the closedness of the range of
	 \[
	 \mathbb{L}^0 : C^\infty(M\times\mathbb{T}^m) \longrightarrow  \mathsf{\Lambda}^{0,1} C^\infty(M\times\mathbb{T}^m),
	 \]
	 where $M$ is assumed to be the interior of a scattering manifold.
	 
	 We begin by recalling that a smooth form  \( f \in \mathsf{\Lambda}^{0,q}C^\infty(M\times\mathbb{T}^m) \) acts on a compactly supported smooth form \( \varphi \in \mathsf{\Lambda}^{0,n-q}C_0^\infty(M\times\mathbb{T}^m) \) by
	 \[
	 \langle f,\varphi\rangle = \int_{M\times\mathbb{T}^m} f\wedge\varphi\wedge \mathrm{d}x,
	 \]
	 where  \(\mathrm{d}x = \mathrm{d}x_1\wedge\cdots\wedge\mathrm{d}x_m.\)
	 
	 Thus, we can identify  \( \mathsf{\Lambda}^{0,q}\mathcal{E}'(M\times\mathbb{T}^m)\) with the topological dual of  \( \mathsf{\Lambda}^{0,n-q}C^\infty(M\times\mathbb{T}^m),\) so that each
	 \(	 \mathsf{\Lambda}^{0,q}\mathcal{E}'(M\times\mathbb{T}^m)\)  is a DFS space.
	 
	 As shown in~\cite{ADL2023gs}, Stokes’ theorem implies that for forms \( f \in \mathsf{\Lambda}^{0,q} C^\infty(M\times\mathbb{T}^m) \)
	 and \( \varphi \in \mathsf{\Lambda}^{0,n-q-1}C^\infty(M\times\mathbb{T}^m) \), one of which has compact support, we have
	 \[
	 \int_{M\times\mathbb{T}^m} (\mathbb{L}^q f)\wedge\varphi\wedge \mathrm{d}x = (-1)^{q+1}
	 \int_{M\times\mathbb{T}^m}  f\wedge(\mathbb{L}^{n-q-1}\varphi)\wedge \mathrm{d}x.
	 \]
	 Consequently,
	 \[
	 \mathbb{L}^q :
	 \mathsf{\Lambda}^{0,q}\mathcal{E}'(M\times\mathbb{T}^m)
	 \longrightarrow
	 \mathsf{\Lambda}^{0,q+1}\mathcal{E}'(M\times\mathbb{T}^m)
	 \]
	 is the transpose of
	 \[
	 (-1)^{q+1}\mathbb{L}^{n-q-1} :
	 \mathsf{\Lambda}^{0,n-q-1}C^\infty(M\times\mathbb{T}^m)
	 \longrightarrow
	 \mathsf{\Lambda}^{0,n-q}C^\infty(M\times\mathbb{T}^m),
	 \]
	 and conversely.
	 
	 In this section we establish a complete characterization of the global solvability of the operator~$\mathbb{L}^0$ and of its adjoint $\mathbb{L}^{n-1}$. 
	 
	 \begin{theorem}\label{thm_solv}
	 	Let $\omega_1,\dots,\omega_m\in \mathsf{\Lambda}^1 C_{\partial M}^\infty(M)$ be a family of real-valued closed $1$-forms on $M$. 
		The following statements are equivalent:
	 	\begin{enumerate}
	 		\item $\mathbb{L}^0$ is globally $(C^\infty,C^\infty)$-solvable;
	 		\item $\mathbb{L}^{n-1}$ is globally $(\mathcal{E}',\mathcal{E}')$-solvable;
	 		\item $\mathbb{L}^{n-1}$ is globally $(\mathcal{E}',C_c^\infty)$-solvable;
	 		\item the family $\boldsymbol{\omega} = (\omega_1,\dots,\omega_m)$ is not $\Gamma_1$-Liouville;
	 		\item the matrix of periods $A(\boldsymbol{\omega})$ satisfies the Diophantine condition~\eqref{DCg} with $\Gamma = \Gamma_1$.
	 	\end{enumerate}
	 \end{theorem}
	 
	 The notions of $\Gamma$-Liouville $1$-forms and of the Diophantine condition~\eqref{DCg} are milder variants of the classical definitions of Liouville $1$-forms and of condition~\eqref{DC}. They will be introduced below in Definitions~\ref{def_gLiou} and~\ref{def_DCg}. The proof of Theorem~\ref{thm_solv} is developed through the results that follow.
	 
	 As is standard, the global $(\mathcal{E}',C_c^\infty)$-solvability of $\mathbb{L}^{n-1}$ yields a Hörmander-type estimate.
	 
	 \begin{proposition}\label{horm_est}
	 	Suppose that $\mathbb{L}^{n-1}$ is globally $(\mathcal{E}',C_c^\infty)$-solvable. Then, for every compact set $K\subset M$, there exist $C>0$ and seminorms $\|\cdot\|_k,\|\cdot\|_\ell$ (depending on $K$) such that
	 	\[
	 	\left| 	\int_{M\times\mathbb{T}^m} \phi(t,x)\, f(t,x)\wedge \mathrm{d}x \right| \leq C\,\|f\|_k\,\|\mathbb{L}^0\phi\|_\ell,
	 	\]
	 	for all $f\in \mathsf{\Lambda}^{0,n}C_c^\infty(K\times\mathbb{T}^m)$ satisfying the compatibility condition~\eqref{solv_EE} and every $\phi\in C^\infty(M\times\mathbb{T}^m)$.
	 \end{proposition}
	 
	 \begin{proof}
	 	The argument follows the lines of \cite[Lemma~6.1.2]{Hormander} and \cite[Lemma~VIII.1.1]{Treves}.
	 	
	 	Let $\mathbf{E}$ denote the subspace of $\mathsf{\Lambda}^{0,n}C_c^\infty(K\times\mathbb{T}^m)$  	consisting of all $f$ satisfying the compatibility condition~\eqref{solv_EE}. Since $\mathbf{E}$ is a closed subspace of $\mathsf{\Lambda}^{0,n}C_c^\infty(K\times\mathbb{T}^m)$, it is a Fréchet space with the induced topology.
	 	
	 	Next, set
	 	\[
	 	\mathbf{F}_0 = C^\infty(M\times\mathbb{T}^m)\cap\ker\mathbb{L}^0, \qquad
	 	\mathbf{F} = C^\infty(M\times\mathbb{T}^m)/\mathbf{F}_0,
	 	\]
	 	and equip $\mathbf{F}$ with the family of seminorms $\|\phi+\mathbf{F}_0\|_\ell := \|\mathbb{L}^0\phi\|_\ell$,
	 	where $\|\cdot\|_\ell$ are the standard seminorms on $C^\infty(M\times\mathbb{T}^m)$. This topology is well defined and makes $\mathbf{F}$ a Fréchet space.
	 	
	 	Define a bilinear form $B:\mathbf{E}\times\mathbf{F}\to\mathbb{R}$ by
	 	\[
	 	B(f,\phi+\mathbf{F}_0) = \int_{M\times\mathbb{T}^m} \phi(t,x)\, f(t,x)\wedge \mathrm{d}x.
	 	\]
	 	The map $B$ is well defined. Indeed, 
	 	if $\phi_1-\phi_2\in\ker\mathbb{L}^0$, then $\langle f,\phi_1-\phi_2\rangle = 0$ by~\eqref{solv_EE},
	 	and hence $B(f,\phi_1+\mathbf{F}_0)=B(f,\phi_2+\mathbf{F}_0)$.
	 	
	 	By \cite[Theorem~34.1]{Treves_TVS}, a bilinear map between locally convex spaces
	 	is continuous whenever it is separately continuous.
	 	Thus, it suffices to check that $B$ is continuous in each argument.
	 	For fixed $\phi\in C^\infty(M\times\mathbb{T}^m)$,
	 	the map $f\mapsto B(f,\phi+\mathbf{F}_0)$ is clearly continuous on~$\mathbf{E}$.	 	
	 
	 	Now fix $f\in\mathbf{E}$. Since $\mathbb{L}^{n-1}$ is globally $(\mathcal{E}',C_c^\infty)$-solvable,
	 	there exists $u\in\mathsf{\Lambda}^{0,n-1}\mathcal{E}'(M\times\mathbb{T}^m)$ such that $\mathbb{L}^{n-1}u=f$. 	Then, for all $\phi\in C^\infty(M\times\mathbb{T}^m)$,
	 	\[
	 	|B(f,\phi+\mathbf{F}_0)| = \Big|\int_{M\times\mathbb{T}^m} 	 	\phi(t,x)\, (\mathbb{L}^{n-1}u) \wedge\mathrm{d}x\Big| = |\langle u,\mathbb{L}^0\phi\rangle| \le C\,\|\mathbb{L}^0\phi\|_\ell,
	 	\]
	 	for some constants $C>0$ and $\ell\in\mathbb{N}$ independent of $\phi$, since $u$ defines a continuous linear functional on $\mathsf{\Lambda}^{0,1}C^\infty(M\times\mathbb{T}^m)$. Thus $B$ is separately continuous, and the desired estimate follows.
	 \end{proof}

	 \begin{lemma}\label{kernel_L0}
	 	If the family $\boldsymbol{\omega}$ is not rational, then $\ker\mathbb{L}^0$ consists only of constant functions.
	 \end{lemma}
	
	 	\begin{proof}
	 	The argument we employ is inspired by the approach in \cite[Lemmas~2.1 and 2.2]{BCP1996}. We split it into two steps.
		
		\noindent
		i) We begin with an auxiliary claim concerning the kernel of a simpler operator.
	 	Let $\alpha\in\mathsf{\Lambda}^1C^\infty(M)$ be a real-valued closed $1$-form, and define the first-order operator 
	 	\[
	 	\mathbb{D}: C^\infty(M)\to\mathsf{\Lambda}^1C^\infty(M),
	 	\qquad
	 	\mathbb{D}u = \mathrm{d}_t u + i\,\alpha\,u,
	 	\]
	 	where we write $\alpha\,u$ instead of $\alpha\wedge u$ since $u$ is scalar-valued. Assume that $\alpha$ is not integral.
	 	We show that $\ker\mathbb{D} = \{0\}$.
	 	
	 	\smallskip
	 	To this aim, let $u\in C^\infty(M)$ satisfy $\mathbb{D}u=0$. Let $\Pi:\hat M\to M$ be the universal covering, 
		and let $\chi\in C^\infty(\hat M)$ be such that $\mathrm{d}\chi = \Pi^*\alpha$. Then
	 	\[
	 	\mathrm{d}_t(\Pi^*u)
	 	= \Pi^*(\mathrm{d}_t u)
	 	= -\,i\,\Pi^*(\alpha u)
	 	= -\,i\,(\mathrm{d}\chi)\,\Pi^*u,
	 	\]
	 	so that
	 	\[
	 	\mathrm{d}_t(e^{i\chi}\Pi^*u)=0.
	 	\]
	 	Hence, $\Pi^*u = c\,e^{-i\chi}$ for some constant $c\in\mathbb{C}$. If $p\in M$ and $A,B\in\hat M$ satisfy $\Pi(A)=\Pi(B)=p$, then
	 	\[
	 	u(p)= c\,e^{-i\chi(A)} = c\,e^{-i\chi(B)}.
	 	\]
	 	If $c\neq 0$, we would obtain $e^{i(\chi(A)-\chi(B))}=1$ for all such points, which by Proposition~\ref{uni_cov}
	 	means that $\alpha$ is integral, a contradiction. Therefore $c=0$, and the claim follows.
	 	
	 	\noindent
	 	ii) Now, let $h\in C^\infty(M\times\mathbb{T}^m)$ satisfy $\mathbb{L}^0h=0$. For each $\xi\in\mathbb{Z}^m$, the Fourier coefficients satisfy
	 	\[
	 	\mathbb{L}_\xi^0\widehat{h}_\xi = \mathrm{d}_t\widehat{h}_\xi + i(\xi\cdot\boldsymbol{\omega})\, \widehat{h}_\xi = 0.
	 	\]
	 	If $\boldsymbol{\omega}$ is not rational, then $\xi\cdot\boldsymbol{\omega}$ is not integral for any $\xi\in\mathbb{Z}^m\setminus\{0\}$. By the previous
		step i), it follows $\widehat{h}_\xi=0$ for all such $\xi$.
	 	
	 	When $\xi=0$, the operator $\mathbb{L}_0^0$ reduces to the exterior derivative $\mathrm{d}_t$ on $M$.
		Since $M$ is connected, any $\widehat{h}_0$ satisfying $\mathrm{d}_t\widehat{h}_0=0$ must be constant. Therefore, $h$ is constant, as claimed.
	 \end{proof}

	\subsection{Global solvability and regularity in subsets of frequencies} \
	
	Motivated by the results of \cite[Section~4]{ADL2023gs}, we investigate the action of $\mathbb{L}^q$
	on subspaces of smooth differential forms whose Fourier series are supported on prescribed subsets of frequencies.
	
	Let $\Gamma\subset\mathbb{Z}^m$ be a subset. For each $q=0,\dots,n$, we define
	\[
	\mathsf{\Lambda}^{0,q}C_{\Gamma}^\infty(M\times\mathbb{T}^m) = \bigl\{ f\in \mathsf{\Lambda}^{0,q}C^\infty(M\times\mathbb{T}^m)\;:\; \widehat{f}_\xi(t)\equiv 0 \text{ for all } \xi\notin\Gamma \bigr\}.
	\]
	
	By Proposition~\ref{cont_fourier}, this is a closed subspace of $\mathsf{\Lambda}^{0,q} C^\infty(M\times \mathbb{T}^m)$. Hence, it is an FS space,
	and carries the natural projection
	\[
	f\longmapsto f_\Gamma = \sum_{\xi\in\Gamma}\widehat{f}_\xi(t)\,e^{i\xi\cdot x}.
	\]
	Moreover, we have the topological direct sum decomposition
	\[
	\mathsf{\Lambda}^{0,q}C^\infty(M\times\mathbb{T}^m) = \mathsf{\Lambda}^{0,q}C_{\Gamma}^\infty(M\times\mathbb{T}^m) \oplus
	\mathsf{\Lambda}^{0,q}C_{\mathbb{Z}^m\setminus\Gamma}^\infty(M\times\mathbb{T}^m).
	\]
	
	Analogously, we define the spaces of distributions $\mathsf{\Lambda}^{0,q} \mathscr{D}'_{\Gamma}(M\times\mathbb{T}^m)$. 
	The operator $\mathbb{L}^q$ then induces continuous linear maps
	\[
	\mathbb{L}_\Gamma^q: \mathsf{\Lambda}^{0,q}C^\infty_{\Gamma}(M\times\mathbb{T}^m) \to
	\mathsf{\Lambda}^{0,q+1}C^\infty_{\Gamma}(M\times\mathbb{T}^m)
	\]
	and
	\[
	\mathbb{L}_\Gamma^q: \mathsf{\Lambda}^{0,q}\mathscr{D}'_{\Gamma}(M\times\mathbb{T}^m)
	\to \mathsf{\Lambda}^{0,q+1}\mathscr{D}'_{\Gamma}(M\times\mathbb{T}^m),
	\]
	for $q=0,\dots,n-1$.
	
%
	\begin{proposition}
		The operator
		\[
		\mathbb{L}^q:
		\mathsf{\Lambda}^{0,q}C^\infty(M\times\mathbb{T}^m)
		\longrightarrow
		\mathsf{\Lambda}^{0,q+1} C^\infty(M\times\mathbb{T}^m)
		\]
		has closed range if and only if both the restricted operators
		\[
		\mathbb{L}_\Gamma^q:
		\mathsf{\Lambda}^{0,q}C^\infty_{\Gamma}(M\times\mathbb{T}^m)
		\to
		\mathsf{\Lambda}^{0,q+1}C^\infty_{\Gamma}(M\times\mathbb{T}^m)
		\]
		and
		\[
		\mathbb{L}_{\mathbb{Z}^m\setminus\Gamma}^q:
		\mathsf{\Lambda}^{0,q} C^\infty_{\mathbb{Z}^m\setminus\Gamma}(M\times\mathbb{T}^m)
		\to
		\mathsf{\Lambda}^{0,q+1}C^\infty_{\mathbb{Z}^m\setminus\Gamma}(M\times\mathbb{T}^m)
		\]
		have closed range.
	\end{proposition}
	
	\begin{proof}
		By Proposition~\ref{cont_fourier}, the map $f\mapsto \widehat{f}_\xi$ is continuous for all $\xi\in\mathbb{Z}^m$. Hence, the subspaces
		$\mathsf{\Lambda}^{0,q}C^\infty_{\Gamma}(M\times\mathbb{T}^m)$ and
		$\mathsf{\Lambda}^{0,q}C^\infty_{\mathbb{Z}^m\setminus\Gamma}(M\times\mathbb{T}^m)$
		are closed, and we have the topological direct sum decomposition
		\[
		\mathsf{\Lambda}^{0,q}C^\infty
		= \mathsf{\Lambda}^{0,q}C^\infty_{\Gamma}
		\oplus
		\mathsf{\Lambda}^{0,q}C^\infty_{\mathbb{Z}^m\setminus\Gamma}.
		\]
		Since $\mathbb{L}^q$ acts diagonally with respect to this decomposition, we obtain
		\[
		\mathrm{ran}(\mathbb{L}^q)
		= \mathrm{ran}(\mathbb{L}_\Gamma^q)
		\oplus
		\mathrm{ran}(\mathbb{L}_{\mathbb{Z}^m\setminus\Gamma}^q),
		\]
		and the statement follows.
	\end{proof}

	\begin{definition}\label{def_gLiou}
		Let $\boldsymbol{\omega}=(\omega_1,\dots,\omega_m)$ be a family of real-valued closed $1$-forms on~$M$. We say that:
		\begin{enumerate}
			\item $\boldsymbol{\omega}$ is \emph{$\Gamma$-rational} if there exists $\xi\in\Gamma$ such that $\xi\cdot\boldsymbol{\omega}$ is an integral $1$-form;
			\item $\boldsymbol{\omega}$ is \emph{$\Gamma$-Liouville} if it is not $\Gamma$-rational and there exist
			a sequence of closed integral $1$-forms $\{\omega_j\}$ and a sequence $\{\xi^{(j)}\}\subset\Gamma$
			with $|\xi^{(j)}|\to\infty$ such that $\{\,|\xi^{(j)}|^j(\xi^{(j)}\cdot\boldsymbol{\omega}-\omega_j)\,\}$ is bounded in
			$\mathsf{\Lambda}^1 C^\infty(M)$.
		\end{enumerate}
	\end{definition}
	
	Clearly, $\Gamma$ must be infinite for a system $\boldsymbol{\omega}$ to be $\Gamma$-Liouville.
	
	\begin{definition}\label{def_DCg}
		A matrix $\boldsymbol{A}\in M_{d\times m}(\mathbb{R})$ satisfies the \emph{Diophantine condition $(\mathrm{DC}_\Gamma)$} if there exist constants $C,\rho>0$ such that 
		\begin{equation}\label{DCg}
			|\eta+\boldsymbol{A}\xi|\geq C(|\eta|+|\xi|)^{-\rho},	\tag{DC$_\Gamma$}
		\end{equation}
		for all $(\xi,\eta)\in\Gamma\times\mathbb{Z}^d\setminus\{(0,0)\}$.
	\end{definition}
	
	By adapting the argument of \cite[Lemma~3.7]{CKMT} to the present setting, we obtain the following characterization.
	
	\begin{lemma}
		If a matrix $\boldsymbol{A}\in M_{d\times m}(\mathbb{R})$ satisfies the Diophantine condition~\eqref{DCg}, then there exist constants $C>0$ and $N\in\mathbb{N}$ such that
		\begin{equation*}
			\max_{1\leq\ell\leq d}|1-e^{2\pi i\xi\cdot a_\ell}|\geq C(1+|\xi|)^{-N},\quad \xi\in\Gamma\setminus\{0\},
		\end{equation*}
		where $a_\ell\in\mathbb{R}^m$ is the $\ell$-th row of $\boldsymbol{A}$.
	\end{lemma}
	
	\begin{proposition}\label{liou_DC}
		The family $\boldsymbol{\omega}$ satisfies:
		\begin{itemize}
			\item[(i)] $\boldsymbol{\omega}$ is $\Gamma$-rational if and only if 
			$A(\boldsymbol{\omega})(\Gamma\setminus\{0\})\cap\mathbb{Z}^d\neq\varnothing$;
			\item[(ii)] $\boldsymbol{\omega}$ is $\Gamma$-Liouville if and only if it is not $\Gamma$-rational and $A(\boldsymbol{\omega})$ does not satisfy~\eqref{DCg}.
		\end{itemize}
	\end{proposition}
	
	\begin{proof}
		The proof is analogous to that of~\cite[Proposition~3.8]{CKMT}.
	\end{proof}
	
	\begin{definition}
		We say that the operator $\mathbb{L}^0$ is \emph{$\Gamma$-globally hypoelliptic} if
		\[
		u\in\mathscr{D}_\Gamma(M\times\mathbb{T}^m),\quad
		\mathbb{L}^0u\in\mathsf{\Lambda}^{0,1}C^\infty(M\times\mathbb{T}^m)
		\ \Rightarrow\ 
		u\in C^\infty(M\times\mathbb{T}^m).
		\]
	\end{definition}
	
	\begin{proposition}\label{gamma_gh}
		If $\boldsymbol{\omega}$ is neither $\Gamma$-rational nor $\Gamma$-Liouville, then 
		$\mathbb{L}_{\Gamma}^0$ is $\Gamma$-globally hypoelliptic. 
		In particular, if $\boldsymbol{\omega}$ is not $\Gamma$-Liouville, then 
		$\mathbb{L}_{\Gamma}^0$ has closed range.
	\end{proposition}
	
	\begin{proof}
		The first assertion follows from the same arguments as in~\cite[Theorem~3.3]{CKMT}. 
		The second one is then a direct consequence of Theorem~\ref{gs_agh_equiv}.
	\end{proof}
	
	\begin{theorem}\label{ECc_horm}
		Suppose that the family $\boldsymbol{\omega}$ is $\Gamma$-Liouville for some subset $\Gamma\subset\mathbb{Z}^m$.  Then the operator $\mathbb{L}^{n-1}$ is not globally $(\mathcal{E}',C_c^\infty)$-solvable.
	\end{theorem}
	
	\begin{proof}
		Assume that $\boldsymbol{\omega}$ is $\Gamma$-Liouville. Then there exist a sequence of closed integral $1$-forms $\{\theta_j\}_{j\in\mathbb{N}}$ and a sequence $\{\xi^{(j)}\}_{j\in\mathbb{N}}\subset\Gamma$ such that $|\xi^{(j)}|\to\infty$ and the family $\{|\xi^{(j)}|^j(\xi^{(j)}\cdot\boldsymbol{\omega}-\theta_j)\}$ is bounded in $\mathsf{\Lambda}^1 C^\infty(M)$.
		
		For each $j\in\mathbb{N}$, let $\psi_j\in C^\infty(\hat M)$ satisfy $\mathrm{d}\psi_j=\Pi^*\theta_j$. Let $\Omega\in\mathsf{\Lambda}^n C^\infty(M)$ be a nonvanishing $n$-form, which exists since $M$ is assumed to be orientable. Fix $\varphi\in C_c^\infty(M)$ with $\varphi\ge 0$ and $\varphi>0$ on the interior of its compact support $K\subset M$.

		Since $e^{i\psi_j}\in C^\infty(M)$ by Proposition~\ref{uni_cov}, define
		\[
		f_j(t,x)=e^{-i(\xi^{(j)}\cdot x-\psi_j(t))}\,\varphi(t)\,\Omega(t)
		\in\mathsf{\Lambda}^{0,n}C_c^\infty(K\times\mathbb{T}^m),
		\]
		and
		\[
		g_j(t,x)=e^{i(\xi^{(j)}\cdot x-\psi_j(t))}\in C^\infty(M\times\mathbb{T}^m).
		\]
		
		By construction, each $f_j$ belongs to 
		$\mathsf{\Lambda}^{0,n}C_c^\infty(K\times\mathbb{T}^m)$. 
		Moreover, from Lemma~\ref{kernel_L0}, we know that 
		$\ker\mathbb{L}^0\cap C^\infty(M\times\mathbb{T}^m)$ 
		consists only of constant functions. Hence,
		\[
		\int_{M\times\mathbb{T}^m} f_j\wedge\mathrm{d}x
		= \left(\int_M e^{i\psi_j}\varphi\Omega\right)
		\left(\int_{\mathbb{T}^m} e^{i\xi^{(j)}\cdot x}\mathrm{d}x\right)
		= 0,
		\]
		which shows that $f_j\in\mathbf{E}$ (as defined in the proof of Proposition~\ref{horm_est}).
		
		Furthermore,
		\begin{equation}\label{int_positive}
			\left|\int_{M\times\mathbb{T}^m} g_j\,f_j\wedge\mathrm{d}x\right|
			= \left|\int_{M\times\mathbb{T}^m}\varphi\,\Omega\wedge\mathrm{d}x\right| > 0.
		\end{equation}
		
		On the other hand, we have
		\[
		\mathbb{L}^0 g_j 
		= -\,i\sum_{k=1}^m (\xi_{j,k}\omega_k - \theta_j)\,
		e^{i(\xi^{(j)}\cdot x - \psi_j(t))}
		\in \mathsf{\Lambda}^{0,1}C^\infty(M\times\mathbb{T}^m).
		\]
		Arguing as in the proof of~\cite[Theorem~3.3]{CKMT}, 
		there exists $C>0$ such that for each $\ell\in\mathbb{N}$,
		\[
		\|\mathbb{L}^0 g_j\|_{K,\ell} \le C\,|\xi^{(j)}|^{-j}.
		\]
		Similarly, differentiating $e^{i\xi^{(j)}\cdot x}$ with respect to $x$ gives
		\[
		\|f_j\|_{K,\ell}\le C\,|\xi^{(j)}|^{\ell}.
		\]
		Hence, for every fixed $\ell\in\mathbb{N}$,
		\[
		\|f_j\|_{K,\ell}\,\|\mathbb{L}^0 g_j\|_{K,\ell}\longrightarrow 0,
		\quad j\to\infty.
		\]
		
		Combining this estimate with~\eqref{int_positive} contradicts the inequality obtained in Lemma~\ref{horm_est}. Therefore, $\mathbb{L}^{n-1}$ cannot be globally $(\mathcal{E}', C_c^\infty)$-solvable.
	\end{proof}

	\subsection{Normal form on integral frequencies} \
	
	In the sequel we consider the subsets of frequencies
	\begin{equation}\label{eq:Gamma01}
	\Gamma_0 = \{\xi\in\mathbb{Z}^m : \xi\cdot\boldsymbol{\omega}\ \text{is integral}\}
	\text{ and }
	\Gamma_1= \mathbb{Z}^m\setminus\Gamma_0.
	\end{equation}
	For each $\xi\in\Gamma_0$, let $\boldsymbol{\psi}_\xi\in C^\infty(\hat M)$ be such that 
	\[
	\mathrm{d}\boldsymbol{\psi}_\xi = \Pi^*(\xi\cdot\boldsymbol{\omega}).
	\]
	Then,
	\[
	\boldsymbol{\psi}_\xi = \sum_{k=1}^{m}\xi_k\psi_k,
	\]
	where each $\psi_k\in C^\infty(\hat M)$ satisfies $\mathrm{d}\psi_k = \Pi^*(\omega_k)$ for $k=1,\dots,m$.
	
	Since $\xi\cdot\boldsymbol{\omega}$ is integral by the definition of $\Gamma_0$, 
	it follows from Proposition~\ref{uni_cov} that $e^{i\boldsymbol{\psi}_\xi}$ descends to a smooth function on~$M$.

	\begin{theorem}\label{top_iso}
		The operator
		\[
		\mathscr{T}:
		\mathsf{\Lambda}^{0,q}C^\infty_{\Gamma_0}(M\times\mathbb{T}^m)
		\longrightarrow
		\mathsf{\Lambda}^{0,q}C^\infty_{\Gamma_0}(M\times\mathbb{T}^m),
		\]
		given by
		\[
		\mathscr{T}(u)(t,x)
		= \frac{1}{(2\pi)^m}
		\sum_{\xi\in\Gamma_0}
		\widehat{u}_\xi(t)\,e^{-i\boldsymbol{\psi}_\xi(t)}e^{i\xi\cdot x},
		\]
		is well defined and is a topological isomorphism.
		Moreover,
		\[
		\mathbb{L}_{\Gamma_0}^q\circ\mathscr{T}
		= \mathscr{T}\circ \mathrm{d}_t.
		\]
	\end{theorem}
	
	\begin{proof}
		The argument follows the approach of the proof of \cite[Proposition~5.1]{AFJR2024}, with suitable adaptations to the present setting. 
		
		Let
		\[
		u(t,x) = \sum_{\xi\in\Gamma_0} \widehat{u}_\xi(t)e^{i\xi\cdot x}
		\in \mathsf{\Lambda}^{0,q}C^\infty_{\Gamma_0}(M\times\mathbb{T}^m),
		\]
		and note that, by continuity of the pullback,
		\[
		\Pi^*u(\hat t,x)
		= \sum_{\xi\in\Gamma_0} \Pi^*\widehat{u}_\xi(\hat t)e^{i\xi\cdot x}
		\in \mathsf{\Lambda}^{0,q}C^\infty(\hat M\times\mathbb{T}^m).
		\]
		
		Define the smooth diffeomorphism \(\Psi : \hat M\times\mathbb{T}^m \to \hat M\times\mathbb{T}^m\) by
		\[
		\Psi(\hat t,x_1,\dots,x_m)
		= (\hat t, x_1-\psi_1(\hat t),\dots,x_m-\psi_m(\hat t)),
		\]
		whose inverse is
		\[
		\Psi^{-1}(\hat t,x_1,\dots,x_m)
		= (\hat t, x_1+\psi_1(\hat t),\dots,x_m+\psi_m(\hat t)).
		\]
		Then,
		\begin{equation}\label{PsiPi_u}
			\Psi^*\Pi^*u(\hat t,x)
			= \sum_{\xi\in\Gamma_0}
			\Pi^*\widehat{u}_\xi(\hat t)
			e^{i(\xi\cdot x - \boldsymbol{\psi}_\xi(\hat t))}.
		\end{equation}
		
		Let us show that the right-hand side of~\eqref{PsiPi_u} descends to $M\times\mathbb{T}^m$ as the form $\mathscr{T}u$. For each $\ell\in\mathbb{N}$, define the truncated sums
		\[
		u_\ell(t,x) = \sum_{|\xi|\le\ell,\,\xi\in\Gamma_0} 	\widehat{u}_\xi(t)e^{i\xi\cdot x},
		\qquad
		w_\ell(t,x) = \sum_{|\xi|\le\ell,\,\xi\in\Gamma_0} \widehat{u}_\xi(t)e^{i(\xi\cdot x - \boldsymbol{\psi}_\xi(t))}.
		\]
		
		Since $u_\ell\to u$ in $\mathsf{\Lambda}^{0,q}C^\infty(M\times\mathbb{T}^m)$, we also have
		\[
		\Psi^*\Pi^*u_\ell \to \Psi^*\Pi^*u 	\quad\text{in}\quad \mathsf{\Lambda}^{0,q}C^\infty(\hat M\times \mathbb{T}^m),
		\]
		and, by construction, $\Pi^*w_\ell = \Psi^*\Pi^*u_\ell$ for all $\ell$.
		
		Fix $t_0\in M$ and choose a neighborhood $U\subset M$ such that $\Pi:\hat U\to U$ is a diffeomorphism
		for some $\hat U\subset\hat M$. Then, the convergence of $\Pi^*w_\ell$ in $\hat U\times\mathbb{T}^m$ 
		implies the convergence of $w_\ell$ in $U\times\mathbb{T}^m$. By estimating in compact subsets, these local limits glue to a global limit in $C^\infty(M\times\mathbb{T}^m)$, showing that $\mathscr{T}$ is well defined.
		The same argument for the inverse transformation $u(t,x)\mapsto (2\pi)^{-m}\sum_{\xi\in\Gamma_0}
		\widehat{u}_\xi(t)e^{i\boldsymbol{\psi}_\xi(t)}e^{i\xi\cdot x}$ proves that $\mathscr{T}$ is a topological isomorphism.
		
		Finally, let us verify the intertwining relation.
		For any $u\in \mathsf{\Lambda}^{0,q}C^\infty_{\Gamma_0}$,
		using that
		$\mathrm{d}_t(e^{-i\boldsymbol{\psi}_\xi})=-i(\xi\cdot\boldsymbol{\omega})
		e^{-i\boldsymbol{\psi}_\xi}$,
		we compute:
		\begin{align*}
			(\mathbb{L}_{\Gamma_0}^q\circ\mathscr{T})u
			&= \Big(\mathrm{d}_t+\sum_{k=1}^m\omega_k\wedge\partial_{x_k}\Big)
			\Big(\sum_{\xi\in\Gamma_0}
			e^{i(\xi\cdot x - \boldsymbol{\psi}_\xi(t))}\widehat{u}_\xi(t)\Big) \\
			&= \sum_{\xi\in\Gamma_0}
			e^{i(\xi\cdot x - \boldsymbol{\psi}_\xi(t))}
			\mathrm{d}_t\widehat{u}_\xi(t)
			= (\mathscr{T}\circ\mathrm{d}_t)u.
		\end{align*}
		This completes the proof.
	\end{proof}

	Although the following result is well-known, we are not aware of a reference containing an explicit proof.  
	For the sake of completeness, we include here a simple argument, based on Stokes’ theorem and de~Rham’s theorem.

	\begin{lemma}\label{extder_closed}
		Let $M$ be a smooth manifold of dimension $n$.  
		For every $q=0,\dots,n-1$, the exterior derivative
		\[
		\mathrm{d} : \mathsf{\Lambda}^q C^\infty(M)
		\longrightarrow \mathsf{\Lambda}^{q+1}C^\infty(M)
		\]
		has closed range.
	\end{lemma}
	
	\begin{proof}
		Let $f$ belong to the closure of $\operatorname{ran}(\mathrm{d})$. Then there exists a sequence $\{u_n\}\subset\mathsf{\Lambda}^q C^\infty(M)$ such that $\mathrm{d}u_n \to f$ in $\mathsf{\Lambda}^{q+1} C^\infty(M)$, where the topology is that of uniform convergence of all derivatives on compact subsets.
		
		By continuity of the exterior derivative,
		\[
		\mathrm{d}f = \lim_{n\to\infty}\mathrm{d}(\mathrm{d}u_n) = 0,
		\]
		so $f$ is closed. Moreover, for any smooth $(q+1)$-cycle $\gamma$ (which is compact), we can pass the limit under the integral sign. Then,
		\[
		\int_\gamma f = \lim_{n\to\infty} \int_\gamma \mathrm{d}u_n
		= \lim_{n\to\infty} \int_{\partial\gamma} u_n = 0,
		\]
		where the last equality follows from Stokes’ theorem. Hence, $f$ is a closed $(q+1)$-form whose periods vanish over all cycles.
		
		By the de~Rham theorem, a closed form with zero periods is exact, thus $f\in \operatorname{ran} (\mathrm{d})$. 
		Therefore, $\operatorname{ran}(\mathrm{d})$ is closed in $\mathsf{\Lambda}^{q+1} C^\infty(M)$.
	\end{proof}
	
	\begin{proposition}\label{partial_closed}
		Let $M$ be a smooth paracompact manifold.  Then the partial exterior derivative
		\[
		\mathrm{d}_t:\mathsf{\Lambda}^{0,q}C^\infty(M\times\mathbb{T}^m) \longrightarrow \mathsf{\Lambda}^{0,q+1}C^\infty(M\times\mathbb{T}^m)
		\]
		has closed range for every $q=0,\dots,n-1$.
	\end{proposition}
	
	\begin{proof}
		This statement is a particular case of \cite[Proposition~1.3.1]{Novelli}. Indeed, by Lemma~\ref{extder_closed}, the exterior derivative 	\(\mathrm{d}: \mathsf{\Lambda}^{q} C^\infty(M)\to\mathsf{\Lambda}^{q+1}C^\infty(M)	\) has closed range for all $q$.  
		
		Moreover, both $\mathsf{\Lambda}^{q}C^\infty(M)$ and $C^\infty(\mathbb{T}^m)$ are Fréchet–nuclear spaces, and we have the natural topological isomorphism
		\[
		\mathsf{\Lambda}^{0,q}C^\infty(M\times\mathbb{T}^m)
		\simeq \mathsf{\Lambda}^{q}C^\infty(M)\,\widehat{\otimes}\, C^\infty(\mathbb{T}^m).
		\]
		Under these conditions, \cite[Proposition~1.3.1]{Novelli} ensures that the operator 
		$\mathrm{d}\otimes I$ on the completed projective tensor product has closed range, which is precisely $\mathrm{d}_t$.
	\end{proof}
	
	\begin{proposition}\label{cl_gamma0}
		The operator
		\[
		\mathbb{L}^0_{\Gamma_0} :
		C^\infty_{\Gamma_0}(M\times\mathbb{T}^m)
		\longrightarrow
		\mathsf{\Lambda}^{0,1} C^\infty_{\Gamma_0}(M\times\mathbb{T}^m)
		\]
		has closed range.
	\end{proposition}
	
	\begin{proof}
		This follows directly from Theorem~\ref{top_iso} and Proposition~\ref{partial_closed}.
	\end{proof}
	
	It remains to obtain conditions for the closedness of the range of
	$\mathbb{L}^0_{\Gamma_1}$, where $\Gamma_1=\mathbb{Z}^m\setminus\Gamma_0$.
	
	\begin{proposition}\label{cl_gamma1}
		If the family $\boldsymbol{\omega}$ is not $\Gamma_1$-Liouville,
		then $\mathbb{L}^0_{\Gamma_1}$ has closed range.
	\end{proposition}
	
	\begin{proof}
		This is a particular case of Proposition~\ref{gamma_gh}.
	\end{proof}
	
	Finally, combining all previous results in this section, we conclude the proof of Theorem~\ref{thm_solv}.
	
	\begin{proof}[Proof of Theorem~\ref{thm_solv}]
		The equivalence $(1)\Leftrightarrow(2)$ follows from Lemma~\ref{lema_gabriel},  
		and the implication $(2)\Rightarrow(3)$ is immediate.  
		The equivalence $(4)\Leftrightarrow(5)$ is given by Proposition~\ref{liou_DC}.  
		The implication $(3)\Rightarrow(4)$ follows from Theorem~\ref{ECc_horm}.  
		Finally, Propositions~\ref{cl_gamma0} and~\ref{cl_gamma1} yield $(4)\Rightarrow(1)$.  
		This completes the proof.
	\end{proof}

	\section*{Acknowledgments}
	
	The authors thank Gabriel Ara\'ujo for the valuable discussions and suggestions.
	This study was financed in part by CAPES -- Brasil (Finance Code 001). The first author was supported in part by the Italian Ministry of the University and Research -- MUR, within the framework of the Call relating to the scrolling of the final rankings of the PRIN 2022 -- Project Code 2022HCLAZ8, CUP D53C24003370006 (PI A.~Palmieri, Local unit Sc.~Resp.~S.~Coriasco). 
	The second and third author were supported in part by CNPq -- Brasil (grants 316850/2021-7 and 301573/2025-5, respectively).

\bibliographystyle{plain}
\bibliography{references}

\begin{thebibliography}{10}

\bibitem{H_loc}
N.~Antoni{\'c} and K.~Burazin.
\newblock On certain properties of spaces of locally {Sobolev} functions.
\newblock In {\em Proceedings of the conference on applied mathematics and
  scientific computing. Selected papers of the third conference on applied
  mathematics and scientific computing, Brijuni, Croatia, June 23--27, 2003.},
  pages 109--120. Dordrecht: Springer, 2005.

\bibitem{Araujo2017}
G.~Ara{\'u}jo.
\newblock Regularity and solvability of linear differential operators in
  {Gevrey} spaces.
\newblock {\em Math. Nachr.}, 291(5-6):729--758, 2018.

\bibitem{ADL2023gs}
G.~Ara{\'u}jo, P.~L.~Dattori da~Silva, and B.~de~Lessa~Victor.
\newblock Global analytic solvability of involutive systems on compact
  manifolds.
\newblock {\em J. Geom. Anal.}, 33(5), 2023.

\bibitem{ADL2023gh}
G.~Ara{\'u}jo, P.~L. Dattori~da Silva, and B.~de~Lessa~Victor.
\newblock Global analytic hypoellipticity of involutive systems on compact
  manifolds.
\newblock {\em Math. Ann.}, 386(3-4):1325--1350, 2023.

\bibitem{AFJR2024}
G.~Ara{\'u}jo, I.~A. Ferra, M.~R. Jahnke, and L.~F. Ragognette.
\newblock Global solvability and cohomology of tube structures on compact
  manifolds.
\newblock {\em Math. Ann.}, 390(2):2199--2233, 2024.

\bibitem{AFR22proc}
G.~Ara{\'u}jo, I.~A. Ferra, and L.~F. Ragognette.
\newblock Global analytic hypoellipticity and solvability of certain operators
  subject to group actions.
\newblock {\em Proc. Am. Math. Soc.}, 150(11):4771--4783, 2022.

\bibitem{AFR2022}
G.~Ara{\'u}jo, I.~A. Ferra, and L.~F. Ragognette.
\newblock Global solvability and propagation of regularity of sums of squares
  on compact manifolds.
\newblock {\em J. Anal. Math.}, 148(1):85--118, 2022.

\bibitem{BCM1993}
A.~P. Bergamasco, P.~D. Cordaro, and P.~A. Malagutti.
\newblock Globally hypoelliptic systems of vector fields.
\newblock {\em J. Funct. Anal.}, 114(2):267--285, 1993.

\bibitem{BCP1996}
A.~P. Bergamasco, P.~D. Cordaro, and G.~Petronilho.
\newblock Global solvability for certain classes of underdetermined systems of
  vector fields.
\newblock {\em Math. Z.}, 223(2):261--274, 1996.

\bibitem{BCH_book}
S.~Berhanu, P.~D. Cordaro, and J.~Hounie.
\newblock {\em An introduction to involutive structures}, volume~6 of {\em New
  Math. Monogr.}
\newblock Cambridge: Cambridge University Press, 2008.

\bibitem{Bourbaki}
N.~Bourbaki.
\newblock {\em Espaces vectoriels topologiques: Chapitres 1 {\`a} 5}.
\newblock Bourbaki, Nicolas. Springer Berlin Heidelberg, 2007.

\bibitem{CD2021}
S.~Coriasco and M.~Doll.
\newblock Weyl law on asymptotically {Euclidean} manifolds.
\newblock {\em Ann. Henri Poincar{\'e}}, 22(2):447--486, 2021.

\bibitem{CKMT}
S.~Coriasco, A.~Kirilov, W.~A.~A. de~Moraes, and P.~M. Tokoro.
\newblock Global hypoellipticity for involutive systems on non-compact
  manifolds.
\newblock {\em J. Geom. Anal.}, 36(22), 2026.

\bibitem{Hormander}
L.~H{\"o}rmander.
\newblock {\em Linear partial differential operators}, volume 116 of {\em
  Grundlehren Math. Wiss.}
\newblock Springer, Cham, 1963.

\bibitem{Komatsu1967}
H.~Komatsu.
\newblock Projective and injective e limits of weakly compact sequences of
  locally convex spaces.
\newblock {\em J.Math. Soc. Japan}, 19:366--383, 1967.

\bibitem{kothe_TVS}
G.~K{\"o}the.
\newblock {\em Topological {V}ector {S}paces. {II}}, volume 237 of {\em
  Grundlehren Math. Wiss.}
\newblock Springer, Cham, 1979.

\bibitem{Kumano-go}
H.~Kumano-go.
\newblock {\em Pseudo-differential operators}.
\newblock {The} {MIT} {Press}, Cambridge, 1981.

\bibitem{Melrose_APS}
R.~B. Melrose.
\newblock {\em The {Atiyah}-{Patodi}-{Singer} index theorem}, volume~4 of {\em
  Res. Notes Math.}
\newblock Wellesley, MA: A. K. Peters, Ltd., 1993.

\bibitem{Melrose_SST}
R.~B. Melrose.
\newblock Spectral and scattering theory for the {Laplacian} on asymptotically
  {Euclidean} spaces.
\newblock In {\em Spectral and scattering theory. Proceedings of the 30th
  Taniguchi international workshop, held at Sanda, Hyogo, Japan}, pages
  85--130. Basel: Marcel Dekker, 1994.

\bibitem{Melrose_GST}
R.~B. Melrose.
\newblock {\em Geometric scattering theory}.
\newblock Cambridge: Cambridge Univ. Press, 1995.

\bibitem{Novelli}
V.~Novelli.
\newblock {\em Regularity and comparison principles in complex analysis and
  locally integrable structures}.
\newblock PhD thesis, Universidade de São Paulo, São Paulo, 2024.

\bibitem{Petersen}
B.~E. Petersen.
\newblock {\em Introduction to the {Fourier} transform and pseudo-differential
  operators}, volume~19 of {\em Monogr. Stud. Math.}
\newblock Pitman, 1983.

\bibitem{Treves_TVS}
F.~Tr\`eves.
\newblock {\em {T}opological {V}ector {S}paces, {D}istributions and {K}ernels}.
\newblock Pure and Applied Mathematics. Academic Press, 1967.

\bibitem{Treves}
F.~Tr\`eves.
\newblock {\em Hypo-analytic structures: local theory}, volume~40 of {\em
  Princeton Math. Ser.}
\newblock Princeton, NJ: Princeton University Press, 1992.

\bibitem{Wloka}
J.~Wloka.
\newblock {\em Partial Differential Equations}.
\newblock Cambridge University Press, 1987.

\end{thebibliography}
	
\end{document}